\documentclass[reqno]{amsart}

\usepackage[left=26mm,top=30mm,right = 26mm, bottom = 30mm]{geometry}

\usepackage[english]{babel} 
\usepackage[utf8]{inputenc} 
\usepackage[T1]{fontenc}
\usepackage{tikz-cd}
\usepackage{enumitem}
\usepackage{amsthm}
\usepackage[backgroundcolor=blue!5!white]{todonotes}
\usepackage{mathtools} 
\usepackage[pagebackref]{hyperref}
\usepackage{mathrsfs}
\usepackage{quiver}
\usepackage{bussproofs}
\usepackage{mleftright}
\usepackage{faktor}
\usepackage{xfrac}
\usepackage{extarrows}
\usepackage{tikz}
\usepackage{stmaryrd}

\usepackage{orcidlink}

\usepackage[capitalize]{cleveref} 

\theoremstyle{plain}

\newtheorem{theorem}{Theorem}[section]

\newtheorem{lemma}[theorem]{Lemma}
\newtheorem{corollary}[theorem]{Corollary}

\newtheorem{claim}[theorem]{Claim}

\theoremstyle{plain}
\newtheorem*{claim*}{Claim}

\theoremstyle{definition}

\newtheorem{definition}[theorem]{Definition}
\newtheorem*{definition*}{Definition}
\newtheorem{notation}[theorem]{Notation}
\newtheorem{example}[theorem]{Example}

\newtheorem{remark}[theorem]{Remark}
\newtheorem*{remark*}{Remark}

\theoremstyle{remark}

\newenvironment{claimproof}[1][Proof of claim]{%
  \begin{proof}[#1]%
}{\end{proof}}

\AddToHook{env/proposition/begin}{\crefalias{theorem}{proposition}}
\AddToHook{env/lemma/begin}{\crefalias{theorem}{lemma}}
\AddToHook{env/corollary/begin}{\crefalias{theorem}{corollary}}
\AddToHook{env/conjecture/begin}{\crefalias{theorem}{conjecture}}
\AddToHook{env/claim/begin}{\crefalias{theorem}{claim}}

\AddToHook{env/definition/begin}{\crefalias{theorem}{definition}}
\AddToHook{env/notation/begin}{\crefalias{theorem}{notation}}
\AddToHook{env/example/begin}{\crefalias{theorem}{example}}
\AddToHook{env/examples/begin}{\crefalias{theorem}{examples}}

\AddToHook{env/remark/begin}{\crefalias{theorem}{remark}}

\newcommand{\cat}[1]{\mathsf{#1}}

\newcommand{\Set}{\cat{Set}}
\newcommand{\BA}{\cat{BA}}
\newcommand{\Pos}{\cat{Pos}}
\newcommand{\Ctx}{\cat{Ctx}}

\newcommand{\C}{\mathsf{C}}

\newcommand{\Ult}{\mathrm{Ult}}

\newcommand{\DoctABA}{\cat{Doct}_{\EA\BA}}
\newcommand{\Stone}{\cat{Stone}}

\newcommand{\LT}{\mathsf{LT}}

\newcommand{\N}{\mathbb{N}}

\newcommand{\T}{\mathcal{T}}

\newcommand{\Hom}{\mathrm{Hom}}

\newcommand{\Mod}{\mathrm{Mod}}
\newcommand{\Typ}{\mathrm{Typ}}

\newcommand{\tmn}{\mathbf{1}}

\newcommand{\F}{\mathbb{F}}

\renewcommand{\P}{\mathbf{P}}

\newcommand{\R}{\mathbf{R}}
\newcommand{\X}{\mathcal{X}}
\newcommand{\Y}{\mathcal{Y}}
\newcommand{\Z}{\mathcal{Z}}

\newcommand{\fa}[2]{(\forall{#1})_{#2}\,}

\newcommand{\ex}[2]{(\exists{#1})_{#2}\,}

\newcommand{\EA}{{\forall\mkern-2mu\exists}}

\newcommand{\ple}[1]{\langle#1\rangle}

\newcommand{\pws}{\mathscr{P}}

\newcommand{\m}{\mathfrak{m}}
\newcommand{\n}{\mathfrak{n}}
\renewcommand{\r}{\mathfrak{r}}

\DeclareMathOperator{\id}{id}
\DeclareMathOperator{\pr}{pr}
\newcommand{\op}{^{\mathrm{op}}}

\newcommand{\longhookrightarrow}{\lhook\joinrel\longrightarrow}

\newcommand{\expl}[2]{%
  \underset{%
    \substack{%
      \big\uparrow\\[-0.2ex]
      \mathrlap{%
        \mbox{%
          \footnotesize
          \hspace{-1em}%
          \begin{tabular}{@{}l@{}}
            #2
          \end{tabular}%
        }%
      }%
    }%
  }{#1}%
}

\newcommand{\equalexpl}[1]{\expl{=}{#1}}

\newcommand{\subseteqexpl}[1]{\expl{\subseteq}{#1}}
\newcommand{\iffexpl}[1]{\expl{\iff}{#1}}

\makeatletter 
\def\l@subsection{\@tocline{2}{0pt}{2pc}{6pc}{}} 
\makeatother

\title[The doctrinal G\"odel's completeness theorem and the type space functor]{The doctrinal G\"odel's completeness theorem\\and the type space functor}

\hypersetup{
    pdftitle={},
    pdfauthor={},
    pdfmenubar=false,
    pdffitwindow=true,
    pdfstartview=FitH,
    colorlinks=true,
    linkcolor=teal,
    citecolor=magenta,
    urlcolor=cyan
}



\author[Marco Abbadini]{Marco Abbadini\textsuperscript{ \orcidlink{0000-0003-1292-6006}}}
\address[Marco Abbadini]{Université catholique de Louvain, Research Institute in Mathematics and Physics, 
1348 Louvain-la-Neuve, Belgium}
\email{marco.abbadini@uclouvain.be}
\urladdr{\url{https://marcoabbadini-uni.github.io}}

\author[Francesca Guffanti]{Francesca Guffanti\textsuperscript{ \orcidlink{0009-0005-8792-0655}}}\address[Francesca Guffanti]{Université Savoie Mont Blanc,
LAMA, Campus Scientifique,
73376 Le Bourget-du-Lac Cedex, France}
\email{francesca.guffanti@univ-smb.fr}
\urladdr{\url{https://sites.google.com/view/francesca-guffanti}}

\begin{document}

\begin{abstract}
    We give a self-contained proof of G\"odel's completeness theorem entirely within the formalism of first-order Boolean doctrines (an algebraic approach to classical many-sorted first-order logic).
    
    Moreover, we show that G\"odel's completeness theorem entails that the fiberwise Stone dual of a first-order Boolean doctrine is its type space functor; roughly speaking, this means that the Stone dual of the Boolean algebra of formulas in context $X$ is the Stone space of $X$-pointed models modulo elementary equivalence.
\end{abstract}

\maketitle

\tableofcontents

\section{Introduction}
Just like an algebraic approach to classical propositional logic is given by Boolean algebras, an algebraic approach to classical first-order logic is given by first-order Boolean doctrines, a notion having its roots in Lawvere's work \cite{Lawvere69,Lawvere70}.
The main idea is to organize formulas in different Boolean algebras according to their contexts; that is, a formula is always considered relative to a finite set of (possibly dummy) free variables.

One of the fundamental theorems of classical first-order logic is G\"odel's completeness theorem \cite{Godel1929,Godel1930} (which can also be phrased as the Model Existence Theorem: every consistent theory has a model).
G\"odel's completeness theorem tells us that the rules of the calculus are the correct ones to capture the intended semantics.

In this paper, we consider a doctrinal version of G\"odel’s completeness theorem, both in the case without and with equality: every consistent first-order Boolean doctrine has a model, and every consistent \emph{elementary} (i.e., with equality)  first-order Boolean doctrine has a model.
These doctrinal versions tell us that the set of axioms of (elementary) first-order Boolean doctrines is complete for the intended semantics. They can be seen as a first-order version of Stone's representation theorem, or of the Boolean prime ideal theorem on which it relies: there are enough models.

The doctrinal G\"odel's completeness theorem in the case without equality follows from a result by Makkai \cite[Thm.~2.1]{Makkai1993} (with the notational difference that he works in the fibrational setting).
In the elementary case, Makkai's fibrational completeness theorem with equality in \cite{Makkai1993} does not apply directly under our assumptions: it requires a finitely complete base and the full Beck--Chevalley condition, whereas we assume only finite products and the Beck--Chevalley conditions associated with quantification and equality. Our assumptions are strictly weaker, even in the Boolean setting; see \cite[Thm.~3.5]{AbbadiniGuffanti-OnBC}.
The doctrinal Gödel’s completeness theorem with equality can nevertheless be recovered indirectly by passing to the associated syntactic Boolean category and applying categorical completeness \cite[Thm.~3.5.5]{MakkaiReyes1977}; see also \cite{Lawvere1967,Volger1975}.
The proofs we present here---both in the case without and with equality---are self-contained and entirely within the doctrinal formalism. We follow Henkin's proof strategy:
one extends the theory to a maximal consistent one, extends the language with witnesses for provable existential statements, and reiterates this process $\omega$ times; the desired model is then obtained by taking the ``algebra of terms'' of the extended theory in the expanded language.
The completeness theorem in the case with equality follows from the one without equality by quotienting a not-necessarily-elementary model by the equivalence relation induced by equality.

As a side technical note, in this paper models are allowed to be empty. 
This reflects the fact that we consider first-order calculus \emph{with contexts}: indeed, the calculus with contexts is well encoded by first-order Boolean doctrines, where the Boolean algebras of formulas are indexed by their set of free variables.

We conclude the paper with a consequence of the completeness theorem for first-order Boolean doctrines: the fiberwise Stone dual of a first-order Boolean doctrine is its \emph{type space functor}. On objects, in the syntactic case this says that the Stone dual of the Boolean algebra of formulas in a context $X$ is the Stone space of $X$-pointed models modulo elementary equivalence.
The underlying idea that the dual of the algebra of first-order formulas is the space of (equivalence classes of) models is contained in some form in the literature; see, e.g., \cite{Makkai1987}, \cite{AwodeyForssell2013}, \cite{Haykazyan2019}, \cite{Kamsma2023}, \cite[Example 4.1]{vanGoolMarques2024}, \cite[Rem.~5.6]{vanGoolWrigley}.
A close formulation to ours is stated, in \cite[Sec.~2.2]{GehrkeJaklEtAL2023}, as a reformulation of G\"odel's completeness theorem. Specializing our theorem---which is formulated for \emph{arbitrary} first-order Boolean doctrines---to a syntactic doctrine recovers it exactly.

\makeatletter
\begingroup
\let\addcontentsline\@gobblethree
\subsection*{Further literature comparison}

A doctrinal version of G\"odel's completeness theorem was obtained by the second author in her doctoral thesis \cite{Guffanti2023}. 
The main difference from this work is that models are required to be non-empty there; accordingly, the hypotheses of G\"odel's completeness theorem were slightly stronger than what we call consistency: while we require only the fiber at the terminal object to be non-trivial in order for a model to exist, \cite{Guffanti2023} requires \emph{all} fibers to be non-trivial.
This difference has an impact on the proofs; indeed, the proof in \cite{Guffanti2023} follows more closely Henkin's original proof, where one first adds to the language infinitely many constants at once (possibly more than needed), and afterwards adds the corresponding witnessing axioms.

A different, dual approach to completeness was recently developed by van Gool and Marquès in the framework of polyadic spaces \cite{vanGoolMarques2024}. 
Their result is more general in certain directions: it applies to doctrines with fibers that are \emph{bounded distributive lattices}, rather than just Boolean algebras, and, moreover, it yields stronger model-theoretic conclusions.
On the other hand, it assumes  the base category to have \emph{all pullbacks} and assumes the \emph{full} Beck--Chevalley condition.
Finally, their proof is formulated dually in terms of polyadic spaces, whereas ours follows the classical Henkin construction more directly.

\endgroup
\makeatother

\section{Preliminaries: first-order Boolean doctrines}

\begin{notation}[Standing notation]\hfill
    \begin{enumerate}
        \item $\N$ denotes the set of natural numbers, including $0$.
        
        \item $\BA$ denotes the category of Boolean algebras and Boolean homomorphisms.
        
        \item $\Pos$ denotes the category of partially ordered sets and order-preserving functions.
        
        \item We let $\tmn_\C$ (or simply $\tmn$) denote the terminal object (when it exists) of a category $\C$.
        For $X \in \C$, we denote by $!_X$ the unique morphism from $X$ to $\tmn$.
    \end{enumerate}
    
\end{notation}

\subsection{First-order Boolean doctrines}

\begin{definition}[First-order Boolean doctrine]\label{d:bool_ex_doc}
    Given a category $\C$ with finite products, a \emph{first-order Boolean doctrine over $\C$} is a functor $\P \colon \C\op \to  \BA$ with the following properties.
    \begin{enumerate}
        \item {(Existential)} \label{i:h3}
        For all $X, Y \in \C$, letting $\pr^{X\times Y}_X \colon X \times Y \to X$ denote the projection onto the first coordinate, the function
        \[
        \P(\pr^{X\times Y}_X) \colon \P(X) \longrightarrow \P(X \times Y),
        \]
        viewed as an order-preserving map between the underlying posets, has a left adjoint $\ex{Y}{X}\!$ (which is not required to be a Boolean homomorphism).
        This means that for every $\alpha \in \P(X \times Y)$ there is a (necessarily unique) element $\ex{Y}{X} \alpha \in \P(X)$ such that, for every $\beta \in \P(X)$, 
        \[
             \ex{Y}{X} \alpha\leq\beta \ \text{ in } \P(X) \quad\iff \quad \alpha \leq \P(\pr^{X\times Y}_X)(\beta)  \ \text{ in }\P(X \times Y).
        \] 
        
        \item
        (Beck--Chevalley) For every morphism $f\colon X'\to X$ in $\C$ and every $Y \in \C$, the following square in $\Pos$ commutes. 
        \[
            \begin{tikzcd}
            {X} & {\P(X\times Y)} & {\P(X)} \\
            X' & {\P(X'\times Y)} & {\P(X')}
            \arrow["{\P(f\times\id_{Y})}", from=1-2, to=2-2, swap]
            \arrow["{\P(f)}", from=1-3, to=2-3]
            \arrow["{\ex{Y}{X'}\!}"', from=2-2, to=2-3]
            \arrow["{\ex{Y}{X}\!}", from=1-2, to=1-3]
            \arrow["f", from=2-1, to=1-1]
            \end{tikzcd}
        \]
    \end{enumerate}
\end{definition}
The category $\C$ is called the \emph{base category of $\P$}.
For $X\in \C$, $\P(X)$ is called the \emph{fiber} over $X$.
For a morphism $f\colon X'\to X$, the function $\P(f) \colon\P(X)\to \P(X')$ is called the \emph{reindexing along $f$}.

\begin{remark}[Existentiality $\Leftrightarrow$ universality]
    In \cref{d:bool_ex_doc}, we required the existence of the left adjoint $\ex{Y}{X}\! \colon \P(X \times Y)\to\P(X)$ of $\P(\pr^{X\times Y}_X)$ (with the Beck--Chevalley condition) for all $X,Y\in\C$.
    In this Boolean case, this condition is equivalent to the existence of the right adjoint $\fa{Y}{X}\!$ (with the Beck--Chevalley condition), as the two quantifiers are interdefinable: $\forall = \lnot \exists \lnot$ and $\exists = \lnot \forall \lnot$.
\end{remark}

\begin{example}[Syntactic doctrine]\label{fbf}
    Let $\T$ be a theory in a first-order language (without equality), and denote by $\F$ the set of function symbols.
    The \emph{syntactic doctrine of $\T$}
    \[
    \LT^{\T} \colon \Ctx_\mathbb{F} \op\longrightarrow \BA
    \]
    (where $\LT$ stands for ``Lindenbaum--Tarski algebra'') is the first-order Boolean doctrine defined as follows. 
    \begin{itemize}
        \item An object of the base category $\Ctx_\mathbb{F}$ is a finite list of distinct variables (also called a \emph{context}).
        \item A morphism in $\Ctx_\mathbb{F}$ from $\vec x=(x_1,\dots, x_n)$ to $\vec y=(y_1,\dots, y_m)$ is an $m$-tuple
        \begin{equation*}
                \bigl(t_1(\vec x),\dots,t_m(\vec x)\bigr)\colon (x_1,\dots, x_n) \longrightarrow (y_1,\dots, y_m)
        \end{equation*}
        of terms in the context $\vec x$.
        The composition of morphisms is given by simultaneous substitutions.
        \item In $\Ctx_\mathbb{F}$, the terminal object is the empty tuple $()$, and the product of two objects $(x_1,\dots,x_n)$ and $(y_1,\dots,y_m)$ is any tuple of distinct variables of length $n+m$, such as $(x_1,\dots,x_n,y_1,\dots,y_m)$ if $(x_1,\dots,x_n)$ and $(y_1,\dots,y_m)$ have no common variables.
        \item On objects, $\LT^{\T} \colon \Ctx_\mathbb{F}\op\to \BA$ maps a context $\vec x = (x_1, \dots, x_n)$ to the poset reflection of the preordered set of formulas whose free variables are within $\{x_1, \dots, x_n\}$, ordered by provable consequence $\vdash_\T$ in $\T$, according to which $\alpha$ is below $\beta$ if and only if the sequent $\alpha \Rightarrow_{\vec x} \beta$ is provable from $\T$; here, the subscript $\vec x$ in the sequent symbol $\Rightarrow$ means that the sequent is considered in the context $\vec x$.\footnote{We refer for example to \cite[Appendix~A]{AbbadiniGuffanti} for the rules of the sequent calculus with contexts for classical first-order logic. A consequence of the slight difference between the calculus \emph{with} contexts and the usual calculus \emph{without} contexts is that the sequent $\Rightarrow_{()} \exists x \top$ is in general not provable in the former, in accordance with admitting the empty set as a possible model.}
        
        \item
        On morphisms, $\LT^{\T}$ maps $\vec{t}(\vec{x}) \colon \vec{x}\to\vec{y}$ to the substitution $[\vec{t}(\vec{x})/\vec{y}] \colon \LT^{\T}(\vec y) \to \LT^{\T}(\vec x)$.
        
        \item Given finite lists of variables $\vec x$ and $\vec y$ with no common variables,
        and letting $\pr_1$ denote the projection morphism $\vec x \colon(\vec x,\vec y)\to\vec x$, the left adjoint to $\LT^{\T}(\pr_1)\colon\LT^{\T}(\vec x)\to\LT^{\T}(\vec x,\vec y)$ (which maps a formula $\alpha(\vec x)$ to itself but with $\vec y$ as dummy variables) is 
        \[
        \exists y_1\,\dots\,\exists y_m\colon\LT^{\T}(\vec x,\vec y)\longrightarrow\LT^{\T}(\vec x).
        \]
    \end{itemize}
\end{example}

\begin{definition}[Morphism of first-order Boolean doctrines] \label{d:ABA-morphism}
    A \emph{first-order Boolean doctrine morphism} from $\P\colon\C\op\to \BA$ to $\mathbf{R} \colon \cat{D}\op\to \BA$ is a pair $(M,\m)$ with $M\colon \C\to\cat{D}$ a functor preserving finite products and $\m \colon \P\to \mathbf{R}\circ M\op $ a natural transformation such that, for all $X,Y\in\C$, the following diagram commutes.
    \[
        \begin{tikzcd}
            \P(X\times Y)\arrow[d,"\ex{Y}{X}\!"']\arrow[r,"\m_{X\times Y}"] & \R\bigl(M(X)\times M(Y)\bigr)\arrow[d,"\ex{M(Y)}{M(X)}\!"]\\
            \P(X)\arrow[r,"\m_{X}"'] & \R(M(X))
        \end{tikzcd}
    \]
    (We implicitly identify $M(X \times Y)$ with $M(X) \times M(Y)$ via the canonical isomorphisms given by the preservation of products by $M$.)
\end{definition}

Given first-order Boolean doctrine morphisms $(M, \m) \colon \P \to \mathbf{R}$ and $(N,\mathfrak n) \colon \R \to \mathbf{S}$,
\[
    \begin{tikzcd}
	\C\op && {\cat{D}\op} && {\cat{E}\op} \\
	\\
	&& \BA
	\arrow["M\op", from=1-1, to=1-3]
	\arrow[""{name=0, anchor=center, inner sep=0}, "\P"', from=1-1, to=3-3]
	\arrow[""{name=1, anchor=center, inner sep=0}, "\R", from=1-3, to=3-3]
	\arrow["N\op", from=1-3, to=1-5]
	\arrow[""{name=2, anchor=center, inner sep=0}, "{\mathbf{S}}", from=1-5, to=3-3]
	\arrow["\m", curve={height=-6pt}, shorten <=8pt, shorten >=8pt, from=0, to=1]
	\arrow["\n", curve={height=-6pt}, shorten <=8pt, shorten >=8pt, from=1, to=2]
    \end{tikzcd}
\]
their composite $(N, \n) \circ (M, \m) \colon \P \to \mathbf{S}$ is the pair $(N \circ M, \mathfrak{n}\circ\m) \colon \P \to \mathbf{S}$, where $N \circ M$ is the composite of the functors between the base categories, and the component at $X\in \C$ of the natural transformation $\mathfrak{n}\circ\m$ is defined as $(\mathfrak{n}\circ\m)_X = \mathfrak{n}_{M(X)}\circ\m_X$, i.e.\ the composite of the following functions:
\[
\P(X)\xrightarrow{\m_X}\mathbf{R}(M(X))\xrightarrow{\mathfrak{n}_{M(X)}}\mathbf{S}(NM(X)).
\]

\begin{notation}[$\DoctABA$]
     We let $\DoctABA$ denote the category of first-order Boolean doctrines and first-order Boolean doctrine morphisms.
\end{notation}

Although 2-categorical aspects of doctrines would be natural, we omit them for simplicity.

\subsection{Models}
 The following additional example of a first-order Boolean doctrine is useful in defining models.

\begin{example}[Subset doctrine]
    The \emph{subset doctrine} is the contravariant power set functor $\pws \colon \Set\op \to \BA$, which maps a set $X$ to its power set $\pws(X)$, and a function $f\colon X'\to X$ to the preimage function
    \[
    \pws(f) \coloneqq f^{-1}[-] \colon \pws(X) \to \pws(X').
    \]
    It is a first-order Boolean doctrine. For $X, Y \in \Set$, the left adjoint $\ex{Y}{X}\!$ to $\pr_X^{-1}[-]\colon \pws(X)\to\pws(X\times Y)$ is the direct image function, which maps $S \in \pws(X\times Y)$ to $\pr_X[S] \in \pws(X)$.
\end{example}

\begin{definition}[Model of a first-order Boolean doctrine]\label{d:univ-bool-model}
    A \emph{model of a first-order Boolean doctrine $\P$} is a first-order Boolean doctrine morphism $(M,\m)$ from $\P$ to the subset doctrine  $\pws$.
    \[
        \begin{tikzcd}
            \C\op && \Set\op \\
            \\
            & \BA
            \arrow["M\op", from=1-1, to=1-3]
            \arrow[""{name=0, anchor=center, inner sep=0}, "\P"', from=1-1, to=3-2]
            \arrow[""{name=1, anchor=center, inner sep=0}, "{\pws}", from=1-3, to=3-2]
            \arrow["\m", curve={height=-6pt}, shorten <=7pt, shorten >=7pt, from=0, to=1]
        \end{tikzcd}
    \]
\end{definition}

\begin{remark}\label{r:model-classical}
    In the syntactic context, a model $(M, \m)$ of $\LT^\mathcal{T}$ corresponds precisely to a model of the theory $\mathcal{T}$ in the classical sense.
    Indeed, the assignment of the functor $M$ on objects encodes the underlying set of the model, the assignment of $M$ on morphisms encodes the interpretation of the function symbols, and the natural transformation $\m$ encodes the interpretation of the relation symbols.
\end{remark}

Note that, in \cref{d:univ-bool-model}, the functor $M$ might assign the empty set to some objects of $\C$. In the syntactic context, this means that we allow the empty model.

\section{G\"odel's completeness theorem}

In this section we prove G\"odel's completeness theorem in doctrinal form (\cref{c:Godel} below).
Calling a first-order Boolean doctrine $\P$ \emph{consistent} if $\top_{\P(\tmn)} \neq \bot_{\P(\tmn)}$, it states
\begin{quote}
    Every consistent first-order Boolean doctrine over a small category has a model.
\end{quote}

\subsection{Henkinness and models}

Our goal is to build a model of a consistent first-order Boolean doctrine. Obtaining a model is easy if the doctrine satisfies a property called \emph{Henkinness} (\cref{d:Henkin}); in this case, one can take as a model the ``algebra of terms'' (\cref{l:Henkin} below).
This is similar to Henkin's proof of G\"odel's completeness theorem (\cite{Henkin1949}).

\begin{notation}[Existential closure]\label{n:global-exists}
    Let $\P\colon \C\op\to \BA$ be a first-order Boolean doctrine.
    For every $X\in\C$ and $\alpha\in \P(X)$, we will be interested in the ``existential closure'' of $\alpha$, which is an element of $\P(\tmn)$. 
    Under the identification $X \cong \tmn \times X$, we can express this element of $\P(\tmn)$ as $\ex{X}{\tmn}\alpha$. In other words, $\ex{X}{\tmn}\!\colon\P(X)\to\P(\tmn)$ is the left adjoint of $\P(!_X)\colon \P(\tmn)\to\P(X)$.
    \[
    \begin{tikzcd}
    	{\P(\tmn\times X)} && {\P(\tmn)} \\
    	\\
    	&& {\P(X)}
    	\arrow[""{name=0, anchor=center, inner sep=0}, "{\ex{X}{\tmn}}", shift left=2, from=1-1, to=1-3]
    	\arrow["\cong"{description}, from=1-1, to=3-3]
    	\arrow[""{name=1, anchor=center, inner sep=0}, "{\P(\pr^{\tmn\times X}_\tmn)}", shift left=2, from=1-3, to=1-1]
    	\arrow[""{name=2, anchor=center, inner sep=0}, "{\P(!_X)}", shift left=2, from=1-3, to=3-3]
    	\arrow[""{name=3, anchor=center, inner sep=0}, "{\ex{X}{\tmn}}", shift left=2, from=3-3, to=1-3]
    	\arrow["\bot"{description}, between={0.2}{0.8}, Rightarrow, from=0, to=1]
    	\arrow["\dashv"{description}, between={0.2}{0.8}, Rightarrow, from=3, to=2]
    \end{tikzcd}
    \]

    Moreover, the Beck--Chevalley condition for a morphism $c\colon \tmn\to X$ in $\C$ and $Y\in\C$ can be rewritten as
    \begin{equation} \label{eq:Beck-Chevalley-for-ex-closure}
        \begin{tikzcd}
        {X} && {\P(X\times Y)} & {\P(X)} \\
        \tmn && {\P(Y)} & {\P(\tmn).}
        \arrow["{\P(\ple{c \circ {!_Y},\id_Y})}", from=1-3, to=2-3, swap]
        \arrow["{\P(c)}", from=1-4, to=2-4]
        \arrow["{\ex{Y}{\tmn}}"', from=2-3, to=2-4]
        \arrow["{\ex{Y}{X}\!}", from=1-3, to=1-4]
        \arrow["c", from=2-1, to=1-1]
        \end{tikzcd}
    \end{equation}

\end{notation}

\begin{definition}[Henkin\footnote{In the syntactic case, our definition corresponds precisely to items i) and ii) at p.~163 in Henkin's paper \cite{Henkin1949}. Indeed, in the syntactic setting, \eqref{i:henkin-1} corresponds to asking that the theory $\T$ is maximal consistent, and \eqref{i:henkin-2} corresponds to asking that for every $\alpha(x)$ such that $\T\vdash \exists x\,\alpha(x)$ there is a closed term $c$ such that $\T\vdash\alpha(c)$.}]\label{d:Henkin}
    A first-order Boolean doctrine $\P \colon \C\op \to \BA$ is \emph{Henkin} if
    \begin{enumerate}
        \item \label{i:henkin-1}
        $\P(\tmn)$ is the two-element Boolean algebra, and
        
        \item \label{i:henkin-2}
        for all $X \in \C$ and $\alpha \in \P(X)$, if $\ex{X}{\tmn} \alpha = \top_{\P(\tmn)}$ then there is $c \colon \tmn \to X$ such that $\P(c)(\alpha) = \top_{\P(\tmn)}$.
    \end{enumerate}
\end{definition}

\begin{remark}\label{r:Henkin-converse}
    The converse of the implication in \eqref{i:henkin-2} in the above definition holds in any first-order Boolean doctrine $\P\colon \C\op \to \BA$, i.e., for
    every $X \in \C$ and $\alpha \in \P(X)$, if there is $c \colon \tmn \to X$ such that $\P(c)(\alpha) = \top_{\P(\tmn)}$, then $\ex{X}{\tmn} \alpha = \top_{\P(\tmn)}$.
    Indeed, by the unit of the adjunction $\ex{X}{\tmn}\dashv \P(!_X)$ at $\alpha$ we have $\alpha \le \P(!_X)\bigl(\ex{X}{\tmn}\alpha\bigr)$.
    Applying $\P(c)$ to this inequality and using functoriality of $\P$,
    \[
    \top_{\P(\tmn)}=\P(c)(\alpha)\le \P(c)\Bigl(\P(!_X)\bigl(\ex{X}{\tmn}\alpha\bigr)\Bigr)
    = \P(!_X\circ c)\bigl(\ex{X}{\tmn}\alpha\bigr)
    = \P(\id_{\tmn})\bigl(\ex{X}{\tmn}\alpha\bigr)
    = \ex{X}{\tmn}\alpha,
    \]
    and hence $\ex{X}{\tmn}\alpha=\top_{\P(\tmn)}$.
\end{remark}

\begin{lemma}[Henkinness gives the ``term'' model] \label{l:Henkin}
    Every Henkin first-order Boolean doctrine $\P \colon \C\op \to \BA$ has a model $(M, \m)$, defined by setting $M$ as $\Hom(\tmn, -) \colon \C \to \Set$ and $\m \colon \P \to \pws\circ M\op$ as the natural transformation whose component at $X \in \C$ is
    \begin{align*}
        \m_X \colon \P(X) & \longrightarrow \pws(\Hom(\tmn,X))\\
        \alpha & \longmapsto \{c \colon \tmn \to X \mid \P(c)(\alpha) = \top_{\P(\tmn)}\}.
    \end{align*}
\end{lemma}

\begin{proof}
    The functor $M \colon \C \to \Set$ clearly preserves finite products (as any representable functor does).
   
    Let us prove that $\m_X$ is a Boolean homomorphism for every $X\in\C$.
    From the fact that $\P(c)$ is a Boolean homomorphism for each $c \colon \tmn \to X$, it is easily seen that $\m_X$ preserves $\top$ and $\land$.
    To prove that $\m_X$ preserves $\lnot$, one observes 
    \[c\in \m_X(\lnot\alpha)\iff \P(c)(\lnot\alpha)=\top \iff \P(c)(\alpha)=\bot \iffexpl{since $\P(\tmn) = 2$} \P(c)(\alpha)\neq \top  \iff c\notin \m_X(\alpha).
    \]
    This concludes the proof that $\m_X$ is a Boolean homomorphism.
    
    The naturality of $\m$ holds because, for all $X, X' \in \C$, $\alpha\in \P(X)$ and all morphisms $f\colon X'\to X$ and $c\colon\tmn\to X'$ in $\C$,
    \begin{align*}
        c\in \m_{X'}\bigl(\P(f)(\alpha)\bigr)
        & \iff \P(c)\bigl(\P(f)(\alpha)\bigr)=\top_{\P(\tmn)}\\
        & \iff\P(f \circ c)(\alpha)=\top_{\P(\tmn)}\\
        & \iff f \circ c\in\m_X(\alpha)\\
        & \iff c\in(f\circ - )^{-1}[\m_{X}(\alpha)].
    \end{align*}

    Let us prove that $\m$ preserves the existential quantifier.
    For $\alpha\in\P(X\times Y)$,
    we shall prove
    \[
    \m_X\bigl(\ex{Y}{X}\alpha\bigr)=\{x \in M(X) \mid \text{there is }y\in M(Y)\text{ s.t.\ } (x,y)\in \m_{X\times Y}(\alpha)\},
    \]
    i.e.
    \begin{equation*}
        \{c \colon \tmn \to X \mid \P(c)(\ex{Y}{X}\alpha) = \top_{\P(\tmn)}\} = \{c \colon \tmn \to X \mid \text{there is }d\colon \tmn \to Y\text{ s.t.\ } \P(\ple{c,d})(\alpha) = \top_{\P(\tmn)}\}.
    \end{equation*}
    (In writing $(x,y)\in \m_{X\times Y}(\alpha)$, we tacitly identify $M(X\times Y)$ and $M(X)\times M(Y)$, i.e., $\Hom(\tmn,X\times Y)$ and $\Hom(\tmn,X)\times\Hom(\tmn,Y)$). This holds because, for every $c \colon \tmn \to X$,
    \begin{align*}
        &\P(c)\bigl(\ex{Y}{X}\alpha\bigr) = \top_{\P(\tmn)} \\
        &\iff \ex{Y}{\tmn}\P(\ple{c \circ {!_Y},\id_Y})(\alpha) = \top_{\P(\tmn)} &&\text{(by Beck--Chevalley as in Not.~\ref{n:global-exists})}\\
        &\iff \text{there is $d\colon \tmn \to Y$ s.t.\ } \P(d)\bigl(\P(\ple{c \circ {!_Y},\id_Y})(\alpha)\bigr)=\top_{\P(\tmn)}&&\text{(by Def.~\ref{d:Henkin} and Rem.~\ref{r:Henkin-converse})}\\
        &\iff \text{there is $d\colon \tmn \to Y$ s.t.\ } \P(\ple{c,d})(\alpha)=\top_{\P(\tmn)}&&\text{(since $\ple{c \circ {!_Y},\id_Y}\circ d=\ple{c,d}$)}. \qedhere 
    \end{align*}
\end{proof}

\subsection{Adding a constant to a doctrine} \label{ss:const}
In \cref{l:Henkin} we saw how to obtain a model out of a Henkin first-order Boolean doctrine.
To get a model for an arbitrary consistent first-order Boolean doctrine $\P$, we produce a Henkin first-order Boolean doctrine out of $\P$.
Since Henkinness requires an abundance of term-definable constants, the process of producing a Henkin first-order Boolean doctrine out of $\P$ requires adding constants to $\P$.
With this in mind, in this subsection we review from \cite{GuffConst} the construction that freely adds to a first-order Boolean doctrine $\P\colon \C\op \to \BA$ a constant of type $S \in \C$.

Let $\P\colon \C\op \to \BA$ be a first-order Boolean doctrine, and $S \in \C$.
Let $\C_S$ be the Kleisli category for the reader comonad $S \times - \colon \C \to \C$:
\begin{itemize}

    \item 
    The objects of $\C_S$ are those of $\C$. 
    
    \item 
    A morphism $f \colon X \rightsquigarrow Y$ in $\C_S$ is a morphism $f \colon S \times X \to Y$ in $\C$. (We use the notation $\rightsquigarrow$ to denote morphisms in $\C_S$.)
    
    \item
    The composite of $f \colon X \rightsquigarrow Y$ and $g \colon Y \rightsquigarrow Z$ is $g \circ \ple{\pr^{S\times X}_S,f} \colon X \rightsquigarrow Z$:
    \[S \times X \xrightarrow{\ple{\pr^{S\times X}_S,f}} S \times Y \xlongrightarrow{g} Z.\]
    
    \item
    The identity on an object $X$ in $\C_S$ is the morphism $X \rightsquigarrow X$ corresponding to the projection onto $X$ in $\C$:
    \[S \times X \xrightarrow{\pr^{S\times X}_X}X.\]
    
\end{itemize}
We remark that, in $\C_S$, there is a morphism $\tmn \rightsquigarrow S$ that corresponds to $\id_S\colon S\to S$ upon choosing $S$ as a product of $S$ and $\tmn$.

The functor $\P_S \colon {\C_S}\op\to\BA$ is defined as follows:
for every $X\in\C_S$, the Boolean algebra $\P_S(X)$ is $\P(S\times X)$; for every morphism $f\colon X\rightsquigarrow Y$ in $\C_S$ (i.e., a morphism $f \colon S \times X \to Y$ in $\C$), the reindexing $\P_S(f)$ is $\P(\ple{ \pr^{S\times X}_S,f })$.
\begin{center}
    \begin{tikzcd}[row sep = 1em]
    \C_S\op\arrow[r, "\P_S"]    &   \BA \\ 
    Y       &   \P(S\times Y)\arrow[dd,"{\P(\ple{ \pr^{S\times X}_S,f })}"]\\ \\
    X\arrow[uu,"f"', squiggly] &     \P(S\times X)
    \end{tikzcd}
\end{center}
For $X,Y\in \C_S$, the existential quantifier $\ex{Y}{X} \colon \P_S(X\times Y)\to \P_S(X)$ in $\P_S$ is the existential quantifier
\[
\ex{Y}{S\times X}\!\colon \P(S\times X\times Y) \longrightarrow \P(S\times X)
\]
in $\P$.

The first-order Boolean doctrine $\P_S$ comes with a canonical first-order Boolean doctrine morphism $(L_S,\mathfrak{l}_S)\colon \P \to \P_S$. The functor $L_S\colon \C \to \C_S$ is the identity on objects and maps a morphism $f \colon X \to Y$ to the morphism $ f\circ \pr^{S\times X}_X\colon X \rightsquigarrow Y$.
The component at an object $X$ of the natural transformation $\mathfrak{l}_S$ is
\begin{align*}
    (\mathfrak{l}_S)_X \colon \P(X) & \longrightarrow \P_S(X)=\P(S\times X)\\
                            \alpha &\longmapsto      \P(\pr^{S\times X}_X)(\alpha).
\end{align*}

All these facts are proved in \cite[Sec.~5]{GuffConst}.

The first-order Boolean doctrine morphism $(L_S,\mathfrak{l}_S)\colon \P \to \P_S$ and the morphism $\id_S\colon \tmn \rightsquigarrow S$ in $\C_S$ have the following universal property.
\begin{theorem}[{\cite[Cor.~6.7 and Thm.~6.3(iii,vii)]{GuffConst}}] \label{t:univ-prop}
    Let $\P$ be a first-order Boolean doctrine over $\C$, let $S\in\C$ and let $(L_S,\mathfrak{l}_S)\colon \P \to \P_S$ be the construction that freely adds a constant of type $S$ to $\P$. 
    For every first-order Boolean doctrine $\mathbf{R}\colon\cat{D}\op\to\BA$, first-order Boolean doctrine morphism $(M,\m)\colon \P\to \mathbf{R}$ and morphism $c \colon \tmn_\cat{D} \to M(S)$ in $\cat{D}$, there is a unique first-order Boolean doctrine morphism $(N,\n) \colon \P_S \to \mathbf{R}$ satisfying $(M,\m)=(N,\n)\circ(L_S,\mathfrak{l}_S)$ and $N(\id_S\colon \tmn \rightsquigarrow S)=c$.
\end{theorem}
(Note that the expression $N(\id_S\colon \tmn \rightsquigarrow S)=c$ type-checks because $L_S$ is the identity on objects and hence $M(S) = N(L_S(S)) = N(S)$.)

\subsection{Adding a (possibly infinite) set of constants to a doctrine}\label{ss:add-infinite-constant}

In the previous subsection, we recalled how to add a single constant.
Here, we explain, given a family of objects $(S_i)_{i\in I}$, how to simultaneously add a constant of type $S_i$ for each $i \in I$.
Drawing on the fact that $I$ is the directed colimit of its finite subsets, the desired construction is obtained as a directed colimit of the construction that adds a constant of type $\prod_{i \in \X} S_i$ for $\X$ ranging among finite subsets of $I$. (Recall that $\C$ has finite products but may lack arbitrary ones.)
This construction already appeared in the second author's doctoral thesis \cite{Guffanti2023}.

Let $(S_i)_{i \in I}$ be a family of objects of $\C$, indexed by a set $I$.
Let $\mathcal{A}$ be the set of finite subsets of $I$, partially ordered by inclusion. The poset $\mathcal{A}$ is directed.
For $\X \in \mathcal{A}$, we write $H_{\X} \coloneqq \prod_{i \in \X} S_i$.

We define an $\mathcal{A}$-shaped diagram in $\DoctABA$:
\[
\begin{tikzcd}
    \mathcal{A} && \DoctABA \\
    {\X} && {{\P}_{H_{\X}}\colon \C_{H_{\X}}\op\to\BA} \\
    {\Y} && {{\P}_{H_{\Y}}\colon \C_{H_{\Y}}\op\to\BA}.
    \arrow[from=1-1, to=1-3, "D"]
    \arrow["\subseteq"{marking}, draw=none, from=2-1, to=3-1]
    \arrow[maps to, from=2-1, to=2-3]
    \arrow[maps to, from=3-1, to=3-3]
    \arrow["{(L_{\X\Y},\mathfrak{l}_{\X\Y})}", from=2-3, to=3-3]
\end{tikzcd}
\] 
\begin{itemize}
    \item (On objects:) for each $\X \in \mathcal{A}$, $D(\X)$ is the first-order Boolean doctrine ${\P}_{H_{\X}}\colon {\C_{H_{\X}}\op\to\BA}$ obtained from ${\P}$ by adding a constant of type $H_{\X}$; in particular, $D(\varnothing) \cong \P$;
    \item (On morphisms:)
    for every $\X,\Y \in \mathcal{A}$ with $\X \subseteq \Y$, $D(\X\subseteq \Y)$ is the unique first-order Boolean doctrine morphism  $(L_{\X\Y},\mathfrak{l}_{\X\Y})\colon \P_{H_{\X}}\to\P_{H_{\Y}}$ such that
    \[    (L_{\X\Y},\mathfrak{l}_{\X\Y})\circ(L_{H_\X},\mathfrak{l}_{H_\X})=(L_{H_\Y},\mathfrak{l}_{H_\Y})\quad\text{ and }\quad L_{\X\Y}(\id_{H_{\X}}\colon \tmn {\overset{H_{\X}}{\rightsquigarrow} }H_{\X})= \pr_{H_{\X}}^{H_{\Y}}\colon \tmn\overset{H_{\Y}}{\rightsquigarrow}H_{\X},
    \]
    which is defined by the universal property of $\P_{H_{\X}}$ (see \cref{t:univ-prop}). We denote by $\overset{S}{\rightsquigarrow}$ an arrow in the category $\C_{S}$.
    \[
    \begin{tikzcd}
    	\P & {\P_{H_\X}} \\
    	& {\P_{H_\Y}}
    	\arrow["{(L_{H_\X},\mathfrak{l}_{H_\X})}"yshift=2pt, from=1-1, to=1-2]
    	\arrow["{(L_{H_\Y},\mathfrak{l}_{H_\Y})}"', from=1-1, to=2-2]
    	\arrow["{(L_{\X\Y},\mathfrak{l}_{\X\Y})}", from=1-2, to=2-2]
    \end{tikzcd}
    \]
    In particular, for every $\X\in\mathcal{A}$, $D(\varnothing\subseteq \X)$ is the canonical first-order Boolean doctrine morphism
    $(L_{H_\X},\mathfrak{l}_{H_\X})\colon {\P}\to \P_{H_{\X}}$. 
\end{itemize}

Let ${\P'}\colon {\C'}\op\to\BA$ be the colimit of the directed diagram $D$ in $\DoctABA$, which is computed as follows (see \cite[Section~1.3]{Guffanti2023}).
\begin{itemize}
    \item (Objects) The objects in the base category $\C'$ are those of $\C$, since for all ${\X}, {\Y} \in \mathcal{A}$ the functor $L_{{\X}{\Y}}$ is the identity on objects.
    
    \item (Morphisms)
    A morphism $[({\X},f)]_{A,B}$ in $\C'$ from $A$ to $B$---written as $[({\X},f)]_{A,B} \colon A\dashrightarrow B$---is the equivalence class of a pair $({\X},f)$ where ${\X}\in \mathcal{A}$ and $f \colon H_{{\X}}\times A \to B$ is a morphism in $\C$, with respect to the equivalence relation $\sim_{A,B}$ defined by setting 
        \[( {\X},f \colon H_{{\X}} \times A \to B) \sim_{A,B} ({\Y},g \colon H_{{\Y}} \times A \to B)
        \]
    if and only if there is ${\Z}\in \mathcal{A}$ with ${\X}\subseteq{\Z}\supseteq{\Y}$ making the following diagram commute.
    \[
        \begin{tikzcd}
            && {H_{ \X}\times A} \\
            {H_{ \Z}\times A} && 
            && B \\
            && {H_{ \Y}\times A}
            \arrow["{\pr_{H_{{\X}}}^{H_{{\Z}}}\times \id_{A}}", from=2-1, to=1-3]
            \arrow["f", from=1-3, to=2-5]
            \arrow["\pr^{H_{ \Z}}_{H_{ \Y}}\times \id_{A}"', from=2-1, to=3-3]
            \arrow["{g}"', from=3-3, to=2-5]
        \end{tikzcd}
    \]
    We will write $[\X,f]_{A,B}$ instead of $[(\X,f)]_{A,B}$, suppressing the parentheses.

    \item (Fibers)
    For every object $A \in \C$, the fiber $\P'(A)$ is the colimit of $D$ in $\BA$ restricted to the fibers.     
    Explicitly, recalling that filtered colimits of Boolean algebras are computed in $\Set$, the Boolean algebra $\P'(A)$ is the set of equivalence classes $[( \X,\alpha)]_{A}$ of pairs $(\X,\alpha)$ where $ \X\in\mathcal{A}$ and $\alpha\in \P(H_{ \X}\times A)$
    with respect to the equivalence relation $\sim_A$ defined by setting
    \[
    {\bigl({\X},\alpha \in \P(H_{ \X}\times A)\bigr)} \sim_A {\bigl({\Y}, \beta \in \P(H_{ \Y}\times A)\bigr)}
    \]
    if and only if there is ${\Z}\in \mathcal{A}$ with ${\X}\subseteq{\Z}\supseteq{\Y}$ such that, in $\P(H_{ \Z}\times A)$,
    \[
    \P(\pr_{H_{{\X}} \times A}^{H_{{\Z}} \times A})(\alpha)=\P(\pr_{H_{{\Y}} \times A}^{H_{{\Z}} \times A})(\beta).
    \]
    We will write $[\X,\alpha]_A$ instead of $[(\X,\alpha)]_A$, suppressing the parentheses.
    
    \item (Reindexings)
    The reindexings are defined using common upper bounds: for $[{{\X}, f}]_{A,B}\colon A\dashrightarrow B$ and $[{{\Y},\psi}]_B\in\P'(B)$, we set
    \begin{equation}\label{eq:def-reindex-p1'}
        \P'([{{\X}, f}]_{A,B})([{\Y}, \psi]_B)=\bigl[{\Z},{\P(\ple{\pr^{H_{ \Z}\times A}_{H_{ \Y}},f\circ\pr^{H_{ \Z} \times A}_{H_{ \X} \times A}})(\psi)}\bigr]_A\in\P'(A),
    \end{equation}
    where ${\Z}$ is any element of $\mathcal{A}$ such that ${\X}\subseteq{\Z}\supseteq{\Y}$ (it does not depend on $\Z$).

    \[\begin{tikzcd}
	&  {H_{ \Z}\times A} & \\
	{H_{ \Z}} && {H_{ \X} \times A} \\
	& {H_{ \Y}\times B} \\
	{H_{ \Y}} && B
	\arrow[from=1-2, to=2-1, "\pr^{H_{ \Z} \times A}_{H_{ \Z}}"']
	\arrow[from=1-2, to=2-3, "\pr^{H_{ \Z} \times A}_{H_{ \X} \times A}"]
	\arrow[from=1-2, to=3-2, "\ple{\pr^{H_{ \Z}\times A}_{H_{ \Y}},f\circ\pr^{H_{ \Z} \times A}_{H_{ \X} \times A}}"{description}]
	\arrow[from=2-1, to=4-1, "\pr^{H_{ \Z}}_{H_{ \Y}}"']
	\arrow[from=2-3, to=4-3, "f"]
	\arrow[from=3-2, to=4-1, "\pr^{H_{ \Y} \times B}_{H_{ \Y}}"]
	\arrow[from=3-2, to=4-3, "\pr^{H_{ \Y} \times B}_{B}"']
\end{tikzcd}\]
    \end{itemize}

The \emph{construction that simultaneously adds to the first-order Boolean doctrine $\P$ a constant of type $S_i$ for each $i \in I$} is the colimit map 
\[
(R,\r)\colon\P\longrightarrow\P'
\]
from $D(\varnothing) = \P$ to the colimit $\P'$.
In particular, recalling that the objects of $\C$ and $\C'$ are the same, for every object $A \in \C$ we have $R(A) = A$, for every morphism $f \colon A \to B$ in $\C$ we have $R(f) = [\varnothing, f]_{A,B}$, and for every $A \in \C$ and $\alpha \in \P(A)$ we have $\r_A(\alpha) = [\varnothing, \alpha]_A$.
Moreover, for every $A \in \C$, $\alpha \in \P(A)$ and ${\X} \in \mathcal{A}$, 
\[
\r_A(\alpha) = \bigl[{\X},\P(\pr_{A}^{H_{{\X}} \times A})(\alpha)\bigr]_A.
\]

\subsection{Reaching Henkinness}
In this subsection, we show how to obtain a Henkin first-order Boolean doctrine from a consistent one.
This blends the process of quotienting by an ultrafilter and adding constants, in line with the two requirements in the definition of Henkinness.

\begin{remark}[{Filter-quotient construction (see, e.g., \cite[Sec.~1.5.1]{Guffanti2023})}]\label{r:quotient-doctrine-filter}
    Let $\P\colon\C\op\to\BA$ be a first-order Boolean doctrine and let $F$ be a filter of the Boolean algebra $\P(\tmn)$.
    Define, for each $X\in\C$, the following preorder on $\P(X)$: for all $\alpha,\beta\in\P(X)$, $\alpha\sqsubseteq^F_X\beta$ if and only if there is $\theta\in F$ such that $\P(!_X)(\theta)\land\alpha\leq\beta$ in $\P(X)$; equivalently, if and only if $\fa{X}{\tmn} (\alpha\to\beta) \in F$.
    Define the first-order Boolean doctrine $\P_{\!/F}\colon\C\op\to\BA$ as follows: for every object $X$, $\P_{\!/F}(X)$ is the poset reflection\footnote{The \emph{poset reflection} of a preordered set $A$ is the quotient $A/{\sim}$ where $a \sim b \Leftrightarrow (a\le b\ \text{and}\ b\le a)$, ordered by $[a]\le [b] \Leftrightarrow a\le b$.} of the preorder $\sqsubseteq^F_X$.
    It can be easily shown that, for every $X\in\C$, $\P_{\!/F}(X)$ is a Boolean algebra, and the quotient map $\pi_X\colon\P(X)\to\P_{\!/F}(X), \,\alpha\mapsto[\alpha]^F_X$ is a Boolean homomorphism.
    For every morphism $f\colon X\to Y$ in $\C$, the function $\P_{\!/F}(f)$ that sends $[\alpha]^F_Y\in \P_{\!/F}(Y)$ to $[\P(f)(\alpha)]^F_X\in\P_{\!/F}(X)$ is well-defined and a Boolean homomorphism.
    For every $X,Y\in\C$ and $[\alpha]^F_{X\times Y}\in\P_{\!/F}(X\times Y)$, it is easy to see that
    $\ex{Y}{X}[\alpha]^F_{X\times Y}\coloneqq[\ex{Y}{X}\alpha]^F_X\in{\P_{\!/F}}(X)$ defines the existential quantification.
    
    It follows that $\P_{\!/F}\colon\C\op\to\BA$ is a first-order Boolean doctrine, and $(\id_{\C},\pi)\colon\P\to\P_{\!/F}$ is a first-order Boolean doctrine morphism.
    
    Observe that the Boolean algebra $\P_{\!/F}(\tmn)$ is the quotient of the Boolean algebra $\P(\tmn)$ by the filter $F$ in the usual sense.
\end{remark}

\begin{lemma}\label{l:universal-join}
    Let $\P$ be a first-order Boolean doctrine over $\C$, let $X,Y\in\C$, let $\alpha,\gamma\in \P(X)$, and let $\beta \in \P(X\times Y)$. Then,
    \[
        \alpha \land  \ex{Y}{X} \beta \leq \gamma \text{ in $\P(X)$} \Longleftrightarrow \P(\pr^{X\times Y}_{X})(\alpha)\land \beta \leq \P(\pr^{X\times Y}_{X})(\gamma) \text{ in $\P(X \times Y)$}.
    \]
\end{lemma}

\begin{proof}
    \begin{align*}
        &\alpha \land  \ex{Y}{X} \beta \leq \gamma&& \text{ in }\P(X)\\ 
        &\Longleftrightarrow   \ex{Y}{X} \beta \leq \alpha \to \gamma && \text{ in }\P(X )\\
        &\Longleftrightarrow  \beta \leq \P(\pr^{X\times Y}_{X})(\alpha \to\gamma)&& \text{ in }\P(X\times Y)\\
        &\Longleftrightarrow  \beta \leq \bigl(\P(\pr^{X\times Y}_{X})(\alpha)\bigr) \to \bigl(\P(\pr^{X\times Y}_{X})(\gamma)\bigr)&& \text{ in }\P(X\times Y)\\
        &\Longleftrightarrow \P(\pr^{X\times Y}_{X})(\alpha)\land \beta \leq \P(\pr^{X\times Y}_{X})(\gamma) &&\text{ in }\P(X \times Y).\qedhere
    \end{align*}
\end{proof}

\begin{lemma}[$\exists$ distributes over $\bigwedge$ with disjoint variables]\label{l:exists-meet}
    Let $\P$ be a first-order Boolean doctrine over $\C$, let $S,X_1,\dots,X_n\in \C$, and let $\alpha_i\in\P(S \times X_i)$ for $i=1,\dots,n$. Then,
    \[
        \bigwedge_{i=1}^n \ex{X_i}{S} \alpha_i
        =
        \ex{\Pi_j X_j}{S}\bigwedge_{i=1}^n \P(\pr^{S\times \Pi_j X_j}_{S \times X_i})(\alpha_i)
    \]
\end{lemma}

\begin{proof}
    For $n=0$, we shall check that $\top_{\P(S)} = \ex{\tmn}{S}\top_{\P(S)}$,
    but this follows from the fact that $\ex{\tmn}{S}$ is the left adjoint of the identity on $\P(S)$, and hence it is the identity.
    
    For $n=1$, the statement is trivially true.
    
   Let $n=2$. We shall obtain the equality
    \begin{equation} \label{eq:binary}
        \ex{X_1}{S} \alpha_1 \land\ex{X_2}{S} \alpha_2 = \ex{X_1\times X_2}{S}\bigl(\P(\pr^{S \times X_1\times X_2}_{S \times X_1})(\alpha_1)\land\P(\pr^{S \times X_1\times X_2}_{S \times X_2})(\alpha_2)\bigr).
    \end{equation}
    By definition of $\exists$, \eqref{eq:binary} holds if and only if for all $\beta\in \P(S)$ we have 
    \[
        \ex{X_1}{S} \alpha_1 \land\ex{X_2}{S} \alpha_2 \leq \beta \iff  \P(\pr^{S \times X_1\times X_2}_{S \times X_1})(\alpha_1)\land\P(\pr^{S \times X_1\times X_2}_{S \times X_2})(\alpha_2)\leq \P(\pr^{S \times X_1\times X_2}_S)(\beta).
    \]
    So, let $\beta \in \P(S)$. We have the following equivalences:
    \begin{align*}
        &\P(\pr^{S \times X_1\times X_2}_{S \times X_1})(\alpha_1)\land\P(\pr^{S \times X_1\times X_2}_{S \times X_2})(\alpha_2)\leq \P(\pr^{S \times X_1\times X_2}_S)(\beta)\\
    &\iff  \P(\pr^{S \times X_1\times X_2}_{S \times X_1})(\alpha_1)\land \P(\pr^{S \times X_1 \times X_2}_{S \times X_2})(\alpha_2)\leq \P(\pr_{S \times  X_1}^{S \times X_1\times X_2})\bigl(\P(\pr^{S \times X_1}_{S})(\beta)\bigr)\\
    &\iff \alpha_1\land \ex{X_2}{S \times X_1}\P(\pr^{S \times X_1 \times X_2}_{S \times X_2})(\alpha_2)\leq \P(\pr^{S \times X_1}_{S})(\beta)&&\text{(by \cref{l:universal-join})}\\
    &\iff \alpha_1\land \P(\pr^{S \times X_1}_{S})\bigl(\ex{X_2}{S}\alpha_2 \bigr)\leq \P(\pr^{S \times X_1}_{S})(\beta)&&\text{(by Beck--Chevalley)}\\
    &\iff \ex{X_1}{S}\alpha_1\land \ex{X_2}{S}\alpha_2\leq \beta &&\text{(by \cref{l:universal-join}).}
    \end{align*}
    
    The statement (for an arbitrary $n$) follows by induction.
\end{proof}

\begin{remark}
Our completeness argument iteratively alternates two operations: completion and
adjoining witnesses for appropriate existential statements.
Iterating this along an $\omega$-chain and taking the colimit yields a Henkin
first-order Boolean doctrine. This is in line with Johnstone's presentation of Gödel's completeness theorem
\cite[Thm.~3.7]{Johnstone1987}. It is also similar to Volger's proof in \cite{Volger1975}.
\end{remark}

\begin{lemma}[Extension to Henkinness: first step]\label{l:ext-Henkin-first-step}
    Let $\P$ be a first-order Boolean doctrine over a small category $\C$ such that $\P(\tmn) = 2$.
    There are a small category $\widetilde{\C}$ with the same objects as $\C$, a first-order Boolean doctrine $\widetilde{\P}$ over $\widetilde{\C}$ such that $\widetilde{\P}(\tmn) = 2$ and a first-order Boolean doctrine morphism $(\widetilde{R},\tilde{\r}) \colon \P \to \widetilde{\P}$ such that $\widetilde{R} \colon \C \to \widetilde{\C}$ is the identity on objects and, for every $X \in \C$ and $\alpha \in \P(X)$ such that $\ex{X}{\tmn} \alpha = \top_{\P(\tmn)}$, there is $c \colon \tmn_{\widetilde{\C}} \to X$ in $\widetilde{\C}$ such that $\widetilde{\P}(c)(\tilde{\r}_X(\alpha)) = \top_{\widetilde{\P}(\tmn)}$.
\end{lemma}
Before starting the proof, we give an informal outline of it.
For every formula $\alpha(X)$ in a context $X$ such that $\exists X\,\alpha(X)$ is true, we add to the language a constant $c_\alpha$ substitutable for $X$, and denote by $\P'$ the doctrine obtained from $\P$ with this expansion of the language.
We denote by $G$ the filter of $\P'(\tmn)$ generated by the formulas $\alpha(c_\alpha/X)$ obtained by substituting $c_\alpha$ for $X$ in $\alpha$, for each $\alpha(X)$ as above. We prove that $G$ is a proper filter, and then apply the Boolean Prime Ideal Theorem to extend $G$ to an ultrafilter $U$, and we conclude by taking the quotient doctrine $\P'_{\!/U}$.

\begin{proof}
    Let 
    \[
    (R, \r) \colon \P \longrightarrow \P'
    \]
    be the construction (detailed in \cref{ss:add-infinite-constant}) that simultaneously adds a constant of type $X$ for each pair $(X,\alpha)$ where $X \in \C$ and $\alpha \in \P(X)$ are such that  $\ex{X}{\tmn} \alpha = \top_{\P(\tmn)}$.
    I.e., letting $I\coloneqq\{(X,\alpha)\mid X \in \C,\, \alpha \in \P(X),\, \ex{X}{\tmn} \alpha = \top_{\P(\tmn)}\}$, we instantiate the construction to the family $(X)_{(X, \alpha) \in I}$.
   
    For every $X \in \C$ and $\alpha \in \P(X)$ such that $\ex{X}{\tmn} \alpha = \top_{\P(\tmn)}$, 
    we let $c_\alpha$ denote the morphism
    \[
    [\{(X,\alpha)\},\id_X]_{\tmn,X} \colon \tmn \dashrightarrow X
    \]
    in $\C'$.
    We let $G$ be the filter of $\P'(\tmn)$ generated by the following subset of $\P'(\tmn)$:
    \[
        \{\P'(c_\alpha)(\r_X(\alpha))\mid X \in \C,\, \alpha \in \P(X)\text{ s.t. }\ex{X}{\tmn} \alpha = \top_{\P(\tmn)}\}.
    \]
    Our next goal is to show that the filter $G$ is proper.

    \begin{claim} \label{f:description-filter}
        An element $\beta' \in \P'(\tmn)$ belongs to ${G}$ if and only if there are $\bar{S}\in\mathcal{A}$ and $\beta \in \P(H_{\bar{S}})$ such that $\beta' = [\bar{S}, \beta]_\tmn$ and
        \begin{equation} \label{eq:witness-filter-membership}
            \bigwedge_{(S, \sigma) \in \bar{S}} \P(\pr_{S}^{H_{\bar{S}} })(\sigma)\leq \beta.
        \end{equation}
    \end{claim}

    \begin{claimproof}
        The filter $G$ is the set of $\beta' \in \P'(\tmn)$ for which there are $n\in\N$, $X_1, \dots, X_n \in \C$ and $\alpha_1 \in \P(X_1), \dots, \alpha_n \in \P(X_n)$ such that for $i=1,\dots,n$ we have $\ex{X_i}{\tmn} \alpha_i = \top_{\P(\tmn)}$ and such that, in $\P'(\tmn)$,
        \[
        \bigwedge_{i=1}^n\P'(c_{\alpha_i})(\r_{X_i}(\alpha_i))\leq\beta',
        \]
        which means, unfolding the definitions, that there is an upper bound $\bar{S} \in \mathcal{A}$ of $\{(X_i,\alpha_i) \mid i \in \{1, \dots, n\}\}$ and $\beta \in \P(H_{\bar{S}})$ such that $\beta' = [\bar{S}, \beta]_\tmn$ and, in $\P(H_{\bar{S}})$
        \begin{equation} \label{eq:unfolded}
            \bigwedge_{i=1}^n \P(\pr^{H_{\bar{S}}}_{X_i})(\alpha_i)\leq \beta.
        \end{equation}
        
        If $\beta'$ possesses this property (i.e.\ $\beta'\in G$), then the formula \eqref{eq:witness-filter-membership} trivially holds since 
        \[
        \bigwedge_{(S, \sigma) \in \bar{S}} \P(\pr_{S}^{H_{\bar{S}} })(\sigma)\leq \bigwedge_{i=1}^n \P(\pr^{H_{\bar{S}}}_{X_i})(\alpha_i) \leq \beta.
        \]
    
        Conversely, let $\bar{S}\in\mathcal{A}$ and $\beta\in \P(H_{\bar S})$ be such that \eqref{eq:witness-filter-membership} holds. 
        Then $[\bar S,\beta]_\tmn\in G$, because \eqref{eq:unfolded} is obtained immediately by taking an enumeration $(X_1, \alpha_1), \dots, (X_n, \alpha_n)$ of the elements of $\bar{S}$.
    \end{claimproof}

    To prove that the filter $G$ is proper, we suppose $\bot_{\P'(\tmn)}\in G$ and we seek a contradiction.
    By \cref{f:description-filter}, there are $\bar{S}\in\mathcal{A}$ and $\beta \in \P(H_{\bar{S}})$ such that $\bot_{\P'(\tmn)} = [\bar{S}, \beta]_\tmn$ and
        \begin{equation} \label{eq:witness-filter-membership1}
            \bigwedge_{(S, \sigma) \in \bar{S}} \P(\pr_{S}^{H_{\bar{S}} })(\sigma) \leq \beta.
        \end{equation}
    Since $\bot_{\P'(\tmn)} = [\bar{S}, \beta]_\tmn$, there is an upper bound $\bar V\in\mathcal{A}$ of $\bar S$ such that
    \begin{equation} \label{eq:equals-bot}
        \P(\pr_{H_{\bar S}}^{H_{\bar V}})(\beta)=\bot_{\P(H_{\bar V})}.
    \end{equation}
    In $\P(H_{\bar V})$ we have:
    \begin{align*}
        \bigwedge_{(V, \nu) \in \bar{V}} \P(\pr_{V}^{H_{\bar{V}} })(\nu)
        &\leq\bigwedge_{(S, \sigma) \in \bar{S}} \P(\pr_{S}^{H_{\bar{V}} })(\sigma)&&\text{(since $\bar S\subseteq\bar V$)}\\
        &\leq  \P(\pr_{H_{\bar S}}^{H_{\bar V}})(\beta) &&\text{(applying $\P(\pr_{H_{\bar S}}^{H_{\bar V}})$ to both sides of \eqref{eq:witness-filter-membership1})}\\
        & = \bot_{\P(H_{\bar{V}})} && \text{(by \eqref{eq:equals-bot})}\\
        &= \P(!_{H_{\bar{V}}})(\bot_{\P(\tmn)}).
    \end{align*}
    By the adjunction $\ex{H_{\bar{V}}}{\tmn}\dashv \P(!_{H_{\bar{V}}})$, we get in $\P(\tmn)$
    \[
        \ex{H_{\bar{V}}}{\tmn}\bigwedge_{(V, \nu) \in \bar{V}}\P(\pr_{V}^{H_{\bar{V}} })(\nu) \leq \bot_{\P(\tmn)}.
    \]

    By \cref{l:exists-meet},
    \[
        \bigwedge_{(V, \nu) \in \bar{V}} \ex{V}{\tmn}\nu \leq \bot_{\P(\tmn)}.
    \]
    Since $\bar{V}\in\mathcal{A}$, for every $(V,\nu)\in\bar V$ we have $\ex{V}{\tmn}\nu=\top_{\P(\tmn)}$, so we get $\top_{\P(\tmn)}\leq\bot_{\P(\tmn)}$, a contradiction (since $\P(\tmn)=2$). Hence, $G$ is a proper filter.
    
    Since $G$ is a proper filter, by the Boolean Prime Ideal Theorem there is an ultrafilter $U$ extending $G$.
    Then, consider the quotient $\P'_{\!/U}$ of $\P'$ by $U$ as in \cref{r:quotient-doctrine-filter}, and let $(\id_{\C'},\pi)\colon\P'\to\P'_{\!/U}$ denote the quotient morphism.
    We define the first-order Boolean doctrine morphism $(\widetilde{R},\tilde{\r}) \colon \P  \to \P'_{\!/U}$ as the composite of the following first-order Boolean doctrine morphisms:
    \[
    \P \xrightarrow{(R,\r)} \P' \xrightarrow{(\id_{\C'},\pi)} \P'_{\!/U};
    \]
    in other words, $(\widetilde{R},\tilde{\r}) \coloneqq (R, \pi \circ \r)$.
    
    We check that the desired properties in the statement are satisfied (upon setting $\widetilde{\P} \coloneqq \P'_{\!/U}$ and $\widetilde{\C} \coloneqq \C'$). 
    First, since $U$ is an ultrafilter, $\P'_{\!/U}(\tmn)=2$. Secondly, observe that $\widetilde{R}$ is the identity on objects since $R$ is.
    Let $\alpha\in\P(X)$ be such that $\ex{X}{\tmn} \alpha = \top_{\P(\tmn)}$.
    We shall find $c \colon \tmn_{\C'} \to X$ in $\C'$ such that $\P'_{\!/U}(c)(\tilde{\r}_X(\alpha)) = \top_{\P'_{\!/U}(\tmn)}$; it is enough to take $c\coloneqq c_\alpha$, because the element $ {\P}'(c_\alpha)({\r}_X(\alpha))$ of $\P'(\tmn)$ belongs to $G$ and so also to $U$,
    which implies that the element $[{\P'}(c_\alpha)({\r}_X(\alpha))]_\tmn = \P'_{\!/U}(c_\alpha)(\tilde{\r}_X(\alpha)) $ is the top element of $\P'_{\!/U}(\tmn)$.
\end{proof}

In the following theorem, we show how to produce a Henkin theory from a maximally consistent one.
We accomplish this by using \cref{l:ext-Henkin-first-step} $\omega$ many times.
The desired Henkin theory is obtained as the colimit.
Note that this Henkin theory is not canonically determined by the original one, because in each step we use the axiom of choice (via the Boolean Prime Ideal Theorem).

\begin{theorem} [Extension to Henkinness] \label{t:extension-to-Henkin}
    Let $\P$ be a first-order Boolean doctrine over a small category $\C$ such that $\P(\tmn) = 2$.
    There are a category ${\C'}$ with the same objects as $\C$, a Henkin first-order Boolean doctrine ${\P'}$ over ${\C'}$ and a first-order Boolean doctrine morphism
    $({R},{\r}) \colon \P \to {\P'}$ such that ${R} \colon \C \to {\C'}$ is the identity on objects.
\end{theorem}

\begin{proof}
    We define a sequence $\P^n \colon (\C^{n})\op \to \BA$ of first-order Boolean doctrines such that $\C^n$ is small and has the same objects as $\C$, together with a sequence $((R^n, \r^n) \colon \P^n \to \P^{n +1})_{n \in \N}$ of first-order Boolean doctrine morphisms where each $R^n\colon \C^n \to \C^{n + 1}$ is the identity on objects, each $\P^n(\tmn)$ is the two-element Boolean algebra $2$, and with the following properties.
    \begin{enumerate}
        \item \label{i:extension-base} 
        $\C^0 = \C$ and $\P^0=\P$;
        
        \item \label{i:extension-Henkinnessinductive}
        for every $X \in \C$ and $\alpha \in \P^n(X)$ with $\ex{X}{\tmn} \alpha = \top_{\P^n(\tmn)}$, there is $c \colon \tmn_{\C^{n + 1}} \to X$ in $\C^{n + 1}$ such that $\P^{n + 1}(c)(\r^n_X(\alpha)) = \top_{\P^{n + 1}(\tmn)}$.
    \end{enumerate}
    
    We define these sequences inductively (with the aid of the axiom of dependent choice).
    For the base case, set $\C^0 \coloneqq \C$ and $\P^0 \coloneqq \P$.
    For the inductive case, for any $n \in \N$, given $\C^n$, $\P^n$, apply \cref{l:ext-Henkin-first-step} to find $\C^{n+1}$, $\P^{n+1}$, and $(R^{n}, \r^n)$.
    This gives us the sequences with the desired properties.
    The sequence allows us to define a directed diagram of first-order Boolean doctrines indexed by the poset $(\N, \leq)$ of natural numbers:
    \begin{equation}\label{eq:diag-n-henkin}
        \P^0
        \xrightarrow{(R^0, \r^0)}
        \P^1
        \xrightarrow {(R^1, \r^1)}
        \P^2
        \longrightarrow
        \dots
        \longrightarrow
        \P^n
        \xrightarrow {(R^n, \r^n)}
        \P^{n+1}
        \longrightarrow
        \dots
    \end{equation}
    For all $n\leq m\in\N$, we call $(R^{n;m},\r^{n;m})$ the composite
    \[
    (R^{m-1},\r^{m-1}) \circ \dots \circ (R^{n+1},\r^{n+1})\circ (R^n,\r^n) \colon \P^n \longrightarrow \P^m.
    \]
    
    We let ${\P'}\colon{\C'}\op\to\BA$ denote the colimit of the diagram \eqref{eq:diag-n-henkin} of first-order Boolean doctrines, which is computed fiberwise (see \cite[Section~1.3]{Guffanti2023}).
    For each $n \in \N$, we let $(E^n,\mathfrak{e}^n) \colon \P^n\to{\P'}$ denote the colimit morphism.
    \[
    \begin{tikzcd}
        {\P^0(X)} & {\P^1(X)} & {\P^2(X)} & \dots & {\P^n(X)} & \dots \\
        \\
        && {{\P'}(X)}
        \arrow["{\r^0_X}", from=1-1, to=1-2]
        \arrow["{\r^2_X}", from=1-3, to=1-4]
        \arrow["{\r^{n-1}_X}", from=1-4, to=1-5]
        \arrow["{\r^n_X}", from=1-5, to=1-6]
        \arrow["{\mathfrak{e}^0_X}"'{description}, from=1-1, to=3-3]
        \arrow["{\r^1_X}", from=1-2, to=1-3]
        \arrow["{\mathfrak{e}^1_X}"'{description}, from=1-2, to=3-3]
        \arrow["{\mathfrak{e}^2_X}"'{description}, from=1-3, to=3-3]
        \arrow["{\mathfrak{e}^n_X}"'{description}, from=1-5, to=3-3]
        \arrow["{\r^{0;2}_X}", curve={height=-18pt}, from=1-1, to=1-3]
    \end{tikzcd}
    \]
    We can choose the colimit in such a way that ${\C'}$ has the same objects as $\C$, and for each $n$ the functor $E^n \colon \C^n \to {\C'}$ is the identity on objects.
    Moreover, we observe that, for all $X \in \C$ and $\beta \in {\P'}(X)$, there are $n \in \N$ and $\alpha \in \P^n(X)$ such that $\beta = \mathfrak{e}^n_X(\alpha)$.

    We prove that ${ \P'}$ is Henkin.
    First of all, we shall prove that ${\P'}(\tmn) = 2$.
    This is true because ${\P'}(\tmn)$ is the directed colimit in the category of Boolean algebras of 
    \[
    2 \longrightarrow 2 \longrightarrow 2 \longrightarrow 2 \longrightarrow \dots
    \]
    In the following, let us denote by ${\exists'}$ and $\exists^n$ the existential quantifiers in the first-order Boolean doctrines ${\P'}$ and $\P^n$ (for every $n\in\N$) respectively.
    Now, let ${\alpha'}\in{\P'}(X)$ be such that $({\exists'}X)_\tmn\, {\alpha'} = \top_{{\P'}(\tmn)}$, and let us prove that there is a morphism ${c'} \colon \tmn \to X$ in ${\C'}$ such that ${\P'}({c'})({\alpha'}) = \top_{{\P'}(\tmn)}$.
    Let $n\in\N$ and $\alpha\in\P^n(X)$ be such that ${\alpha'}=\mathfrak{e}^n_{X}(\alpha)$.
    We have
    \begin{equation} \label{eq:top-is-image-of-exists}
        \top_{{\P'}(\tmn)} = ({\exists'}X)_\tmn\, {\alpha'} = ({\exists'}X)_\tmn\, \mathfrak{e}_{X}^n(\alpha) = \mathfrak{e}_\tmn^n\bigl((\exists^n X)_\tmn\, \alpha\bigr).
    \end{equation}
    Since $\mathfrak{e}_\tmn^n$ is a Boolean homomorphism from $\P^n(\tmn) = 2$ to ${\P'}(\tmn) = 2$, it is the identity, and so from \eqref{eq:top-is-image-of-exists} we deduce $(\exists^n X)_\tmn\, \alpha = \top_{\P^n(\tmn)}$.
    Then, by \eqref{i:extension-Henkinnessinductive}, there is $c \colon \tmn_{\C^{n + 1}} \to X$ in $\C^{n + 1}$ such that 
    \begin{equation} \label{eq:reindexing-is-top}
        \P^{n + 1}(c)(\r^n_X(\alpha)) = \top_{\P^{n + 1}(\tmn)}.
    \end{equation}
    Therefore,
    \begin{align*}
        & \bigl({\P'}(E^{n + 1}(c))\bigr)({\alpha'}) \\
        & = \bigl({\P'}(E^{n + 1}(c))\bigr)\bigl(\mathfrak{e}_{X}^{n}(\alpha)\bigr)\\
        & = \bigl({\P'}(E^{n + 1}(c))\bigr)\Bigl(\mathfrak{e}_{X}^{n + 1}\bigl(\r^n_X(\alpha)\bigr)\Bigr)\\
        & = \mathfrak{e}^{n + 1}_\tmn\Bigl(\P^{n + 1}(c)\bigl(\r^n_X(\alpha)\bigr)\Bigr) && \text{(by naturality of $\mathfrak{e}^{n+1} \colon \P^{n + 1} \to \P'\circ (E^{n + 1})\op$ at $c \colon \tmn \to X$)} \\
        & = \mathfrak{e}^{n + 1}_\tmn(\top_{\P^{n + 1}(\tmn)}) && \text{(by \eqref{eq:reindexing-is-top})}\\
        & = \top_{{\P'}(\tmn)}.
    \end{align*}
    Taking $c' \coloneqq E^{n + 1}(c)$ gives the desired result.
    This shows that $\P'$ is Henkin.
    
    Finally, for the required first-order Boolean doctrine morphism $({R},{\r}) \colon \P \to{ \P'}$ in the statement we can take $(E^0,\mathfrak{e}^0)$. 
\end{proof}

\subsection{G\"odel's completeness theorem}

\begin{notation}[The satisfaction relation $\vDash$]\label{n:semantic-conseq}
    Let $(M, \m)$ be a model of a first-order Boolean doctrine $\P$.
    Since $M$ preserves finite products, $M(\tmn)$ is a singleton $\{*\}$.
    Let $\alpha \in \P(\tmn)$.
    Note that $\m_\tmn(\alpha)$ (which is a subset of $\{*\}$) is either $\{*\}$ or $\varnothing$.
    We write $(M, \m) \vDash \alpha$ if $\m_\tmn(\alpha)$ is the singleton $\{*\}$ (``the interpretation of $\alpha$ in $(M, \m)$ is true'').
\end{notation}

\begin{theorem}[Order separation via models]
\label{t:Godel}
    Let $\P$ be a first-order Boolean doctrine over a small category.
    If $\alpha \nleq \beta$ in $\P(\tmn)$, then $\P$ has a model $(M, \m)$ such that $(M, \m) \vDash \alpha$ and $(M, \m) \nvDash \beta$.
\end{theorem}

\begin{proof}
    By the Boolean Prime Ideal Theorem, since $\alpha \nleq \beta$ in $\P(\tmn)$, there is an ultrafilter $U$ of $\P(\tmn)$ containing $\alpha$ and not $\beta$.
    Let $\C$ be the base category of $\P$, and let $(\id_{\C},\pi)\colon\P\to\P_{\!/U}$ be the quotient of $\P$ by the ultrafilter $U$ defined as in \cref{r:quotient-doctrine-filter}. 
    
    Since $U$ is an ultrafilter, ${\P_{\!/U}}(\tmn) = 2$. Moreover, $\pi_\tmn(\alpha)=\top$ and $\pi_\tmn(\beta)=\bot$.
    By \cref{t:extension-to-Henkin}, there are a category ${\C'}$ with the same objects as the base category $\C$ of ${\P_{\!/U}}$, a Henkin first-order Boolean doctrine ${\P'}$ over ${\C'}$ and a first-order Boolean doctrine morphism
    $({R},{\r}) \colon \P_{\!/U} \to {\P'}$ such that ${R} \colon \C \to {\C'}$ is the identity on objects.
    By \cref{l:Henkin}, ${\P'}$ has a model $(N,\n)$. Then, the composite $(M,\m)=(N,\n)\circ(R,{\r})\circ(\id_\C,\pi)\colon\P\to\pws$ is a model of $\P$. Moreover, 
    \begin{align*}
    \m_\tmn (\alpha)&=\n_\tmn(\mathfrak{r}_\tmn(\pi_\tmn(\alpha)))=\n_\tmn(\mathfrak{r}_\tmn(\top))=\{*\}&&\text{i.e.\ $(M, \m) \vDash \alpha$, and}\\
    \m_\tmn (\beta)&=\n_\tmn({\mathfrak{r}}_\tmn(\pi_\tmn(\beta)))=\n_\tmn(\mathfrak{r}_\tmn(\bot))=\varnothing && \text{i.e.\ $(M, \m) \nvDash \beta$.} \qedhere
    \end{align*}
\end{proof}

Recalling that we call a first-order Boolean doctrine $\P$ \emph{consistent} if $\top_{\P(\tmn)} \neq \bot_{\P(\tmn)}$, we finally obtain:

\begin{corollary}[G\"odel's completeness theorem for first-order Boolean doctrines] \label{c:Godel}
	Every consistent first-order Boolean doctrine over a small category has a model.
\end{corollary}

\section{Generalization to open formulas and pointed models}

\begin{definition}[Model at an object of a first-order Boolean doctrine] \label{d:bool-mod-at-S}
    Let $\P \colon \C\op \to  \BA$ be a first-order Boolean doctrine and let $S \in \C$.
    An \emph{$S$-pointed model of $\P$} is a triple $(M,\m,s)$ where $(M,\m)\colon \P \to \pws$ is a model of $\P$ (i.e.\ a first-order Boolean doctrine morphism from $\P$ to the subset doctrine $\pws$) and $s \in M(S)$.
\end{definition}

Roughly speaking, an $S$-pointed model of $\P$ consists of a model of $\P$ and a value assignment of $S$ in the model.
Up to a one-to-one correspondence (\cref{t:univ-prop}), this is the same thing as a model of the first-order Boolean doctrine $\P_S$ obtained from $\P$ by adding a constant of type $S$. 
\cref{l:model-model-at-S} below asserts the equivalence between the validity of a sentence in $\P_S$ in a model of $\P_S$ and the validity of the associated formula in $\P$ in the associated $S$-pointed model of $\P$.
We make use of the following notation.

\begin{notation}[The satisfaction relation $\vDash$]\label{n:semantic-conseq-at-S}
    Let $\P \colon \C\op \to \BA$ be a first-order Boolean doctrine, let $S \in \C$, let $(M, \m, s)$ be an $S$-pointed model of $\P$, and let $\alpha \in \P(S)$.
    We write $(M, \m, s) \vDash \alpha$ if $s \in \m_S(\alpha)$ (``the interpretation of $\alpha$ in $(M, \m,s)$ is true'').
\end{notation}

\begin{lemma}\label{l:model-model-at-S}
    Let $\P_S$ be the first-order Boolean doctrine obtained from a first-order Boolean doctrine $\P\colon \C\op \to \BA$ by adding a constant of type $S \in \C$, and let $(L_S,\mathfrak{l}_S)\colon \P \to \P_S$ be the canonical first-order Boolean doctrine morphism. Let $(M,\m,s)$ be an $S$-pointed model of $\P$ and $(N,\n)\colon \P_S\to\pws$ the unique model of $\P_S$ such that $(M,\m)=(N,\n)\circ(L_S,\mathfrak{l}_S) $ and $N(\id_S\colon \tmn \rightsquigarrow S)\colon \{*\} \to M(S)$ maps $*$ to $s \in M(S)$. For every $\alpha\in \P_S(\tmn)=\P(S)$,
    \[
    (N,\n)\vDash \alpha \iff (M,\m,s)\vDash\alpha.
    \]
\end{lemma}

\begin{proof}
    This is a particular case of {\cite[Lem.~2.22]{AbbadiniGuffanti2}}.
\end{proof}

We can strengthen G\"odel's completeness theorem for first-order Boolean doctrines as follows.

\begin{corollary}[of \cref{t:Godel}; order-separation via models, at any fiber]
    \label{c:godel-at-s}
    Let $\P$ be a first-order Boolean doctrine over a small category $\C$.
    For all $S \in \C$ and $\alpha, \beta \in \P(S)$ such that $\alpha \nleq \beta$, there is an $S$-pointed model $(M, \m, s)$ of $\P$ such that $(M, \m, s) \vDash \alpha$ and $(M, \m, s) \nvDash \beta$.
\end{corollary}

\begin{proof}
    Let $\P$ be a first-order Boolean doctrine over a small category $\C$, $S\in\C$, $\alpha, \beta \in \P(S)$ such that $\alpha \nleq \beta$. Let $\P_S$ be the first-order Boolean doctrine obtained from $\P$ by adding a constant of type $S$, and let $(L_S,\mathfrak{l}_S)\colon \P \to \P_S$ be the canonical first-order Boolean doctrine morphism.
    
    Since $\alpha \nleq \beta$ in $\P_S(\tmn)=\P(S)$, by \cref{t:Godel} there is a model $(N,\n)$ of $\P_S$ such that $(N,\n)\vDash\alpha$ and $(N,\n)\nvDash\beta$. By the universal property of $(L_S,\mathfrak{l}_S)$ (\cref{t:univ-prop}), the model $(N,\n)$ of $\P_S$ corresponds to an $S$-pointed model $(M,\m,s)$ of $\P$, where $(M, \m)$ is the composite $(N,\n)\circ(L_S,\mathfrak{l}_S)$, and $s\in N(S)$ is the value at $*$ of the function $N(\id_S\colon \tmn\rightsquigarrow S)\colon \{*\}\to N(S)$.
    By \cref{l:model-model-at-S}, since 
    $(N,\n)\vDash\alpha$ we have $(M, \m, s) \vDash \alpha$, and since $(N,\n)\nvDash\beta$ we have $(M, \m, s) \nvDash \beta$.
\end{proof}

\section{The case with equality}

In this section, we consider what happens if the language has equality.
In brief, the appropriate version of G\"odel's completeness theorem still holds, and the proof works the same, except that, in the construction of the ``term'' model of a Henkin elementary first-order Boolean doctrine, one has to identify provably equal terms.

\subsection{Preliminaries: elementary first-order Boolean doctrines}

We start by recalling a modern definition of \emph{elementarity} (i.e.\ the possibility of interpreting equality) in the context of first-order Boolean doctrines.
This notion, whose roots are in Lawvere's work \cite{Lawvere70}, has been introduced by Maietti and Rosolini in \cite{MaiettiRosolini13} in the more general context of primary doctrines.
In this paper, instead of Maietti and Rosolini's definition, we use an equivalent condition presented by Emmenegger, Pasquali and Rosolini in \cite[Prop.~2.5]{EmPaRo20}.

\begin{definition}[Elementary first-order doctrine]\label{def:equality}
    A first-order Boolean doctrine $\P\colon\C\op\to\BA$ is \emph{elementary} if there is a family \[
    \Big(\delta_X\in\P(X\times X)\Big)_{X\in\C}
    \]
    such that, for all $X,Y\in\C$,
    \begin{enumerate}
        \item\label{i:def=1} (Reflexivity)
        denoting by $\Delta_X \colon X \to X \times X$ the diagonal $\ple{\id_X, \id_X}$, 
        \[
        \top_{\P(X)}\leq\P(\Delta_X)(\delta_X) \quad \text{in $\P(X)$};
        \]
        
        \item\label{i:def=2} (Substitutivity)
        denoting by $\pr_1, \pr_2 \colon X\times X \to X$ the two projections, for every $\alpha\in\P(X)$,
        \[
        \P(\pr_1)(\alpha)\land \delta_X\leq \P(\pr_2)(\alpha)\quad \text{in $\P(X\times X)$};
        \]
        
        \item\label{i:def=3} (Product equality) denoting by $\pr_1, \pr_2, \pr_3,\pr_4$ the four projections from $X\times Y\times X\times Y$ to the factors,
        \[
        \Bigl(\P(\ple{\pr_1,\pr_3})(\delta_X)\Bigr)\land \Bigl(\P(\ple{\pr_2,\pr_4})(\delta_Y)\Bigr)\leq \delta_{X\times Y}\quad \text{in $\P(X\times Y\times X\times Y)$}.
        \]
    \end{enumerate}
    For each $X\in\C$, the element $\delta_X$ is called the \emph{fibered equality on $X$}.
\end{definition}

Informally, condition \eqref{i:def=1} in \cref{def:equality} is the reflexivity of the equality relation, i.e.\ $\vdash x=x$. Condition \eqref{i:def=2} is the substitutivity property, i.e.\ $\alpha(x)\land (x=x')\vdash \alpha(x')$. Finally, condition \eqref{i:def=3} can be roughly interpreted as ``$(x=x')\land (y=y')\vdash (x,y)=(x',y')$'', meaning that two pairs coincide if both entries do.

\begin{remark}[{Uniqueness of the fibered equality, \cite[Rem.~2.2(a)]{MaiettiRosolini13}}]
    For every first-order Boolean doctrine $\P \colon \C\op \to \BA$, there is at most one family $(\delta_X)_{X\in\C}$ satisfying the conditions in \cref{def:equality}.
\end{remark}

\begin{remark}[Symmetry of equality, {\cite[p.~450]{MaiettiRosolini13}, or \cite[p.~32]{Pasquali15}}]\label{r:delta-symm}
    For every elementary first-order Boolean doctrine $\P$ and every object $X$ in its base category we have in $\P(X \times X)$
    \[
    \delta_X \leq \P(\ple{\pr_2,\pr_1})(\delta_X),
    \]
    which informally states the symmetry of the equality relation, i.e.\ $x=x'\vdash x'=x$.
\end{remark}

\begin{remark}[Transitivity of equality, {\cite[p.~450]{MaiettiRosolini13}}] \label{r:trans}
    In $\P(X\times X\times X)$ we have
    \[
    \Bigl(\P(\ple{\pr_1,\pr_2})(\delta_X)\Bigr)\land \Bigl(\P(\ple{\pr_2,\pr_3})(\delta_X)\Bigr)\leq \P(\ple{\pr_1,\pr_3})(\delta_X),
    \]
    which informally states the transitivity of the equality relation, i.e.\ $(x=x')\land (x'=x'')\vdash(x=x'')$.
\end{remark}

\begin{remark}[{\cite[Rem.~2.3]{MaiettiRosolini13}}]\label{rmk:aeq-3-conv}
Condition \eqref{i:def=3} in \cref{def:equality} is an equality for every elementary first-order Boolean doctrine, i.e.\ in $\P(X \times Y \times X \times Y)$ we have
\[
 \delta_{X\times Y}= \Bigl(\P(\ple{\pr_1,\pr_3})(\delta_X)\Bigr)\land \Bigl(\P(\ple{\pr_2,\pr_4})(\delta_Y)\Bigr).
\]
\end{remark}

\begin{remark}[Substitutivity for terms, {\cite[Rem.~2.2(d)]{MaiettiRosolini13}}]\label{rmk:aeq}
    Let $\P\colon\C\op\to\BA$ be an elementary first-order Boolean doctrine and let $f\colon X\to Y$ be a morphism in $\C$. We have in $\P(X \times X)$:
    \[
        \delta_X \leq \P(f\times f)(\delta_Y).
    \]
    Informally, this states the substitution property for terms: $x=x'\vdash f(x)=f(x')$. 
\end{remark}

\begin{example}[Syntactic doctrine]\label{fbf=}
    For a first-order theory $\T$ in a language with equality, one can define a syntactic doctrine $\LT_{=}^\T$ similarly to \cref{fbf}. In this case, one has to consider also the formulas involving equality, and take the equivalence relation of equiprovability modulo $\T$ in first-order logic with equality.
    The syntactic doctrine $\LT_{=}^\T$ is an elementary first-order Boolean doctrine: for every context $\vec x = (x_1,\dots,x_n)$, the fibered equality $\delta_{\vec x}\in \LT_=^\T(x_1,\dots,x_n,x'_1,\dots,x'_n)$ is the formula $(x_1=x'_1)\land\dots\land (x_n=x'_n)$.
\end{example}

\begin{example}[Subset doctrine]
    The subset doctrine $\pws\colon\Set\op\to\BA$ is elementary: the fibered equality $\delta_X\in \pws(X\times X)$ on a set $X$ is the equality relation $\Delta_X=\{(x,x)\mid x \in X\}\subseteq X\times X$.
\end{example}

\begin{definition}[Morphism of elementary first-order doctrines]
    An \emph{elementary first-order Boolean doctrine morphism} from $\P\colon\C\op\to \BA$ to $\R \colon \cat{D}\op\to \BA$ is a first-order Boolean doctrine morphism $(M,\m)\colon \P\to\R$ that preserves fibered equalities, i.e.\ such that, for every $X\in\C$, $\m_{X\times X}(\delta^\P_X)=\delta^\R_{M(X)}$.
\end{definition}

\begin{definition}[Model of an elementary first-order Boolean doctrine]\label{d:univ-bool-model-elementary}
    A \emph{model of an elementary first-order Boolean doctrine $\P$} is an elementary first-order Boolean doctrine morphism $(M,\m)$ from $\P$ to the subset doctrine  $\pws$.
\end{definition}

\subsection{G\"odel's completeness theorem, with equality}

As in \cref{n:semantic-conseq}, for a model $(M, \m)$ of an elementary first-order Boolean doctrine $\P$ and $\alpha \in \P(\tmn)$, we write $(M, \m) \vDash \alpha$ if $\m_\tmn(\alpha)$ is the whole singleton $M(\tmn)$.

\begin{theorem} \label{t:quotient}
    Let $\P$ be an elementary first-order Boolean doctrine.
    Let $(M, \m)$ be a model of $\P$ as a (not necessarily elementary) first-order Boolean doctrine.
    There is a model $(M', \m')$ of $\P$ as an \emph{elementary} first-order Boolean doctrine such that, for every $\alpha \in \P(\tmn)$, $(M, \m) \vDash \alpha$ if and only if $(M', \m') \vDash \alpha$.
\end{theorem}

\begin{proof}
    The idea is to take the quotient of $(M,\m)$ that identifies elements related by the interpreted equality relation.
    
    The functor $M'$ is defined as follows:
    \begin{itemize}
        \item (On objects:) for every object $X$ in the base category $\C$, $M'(X) \coloneqq M(X)/{\sim_X}$, where ${\sim_X}$ is the equivalence relation on $M(X)$ defined by: 
        \[
        x \sim_X y \iff (x,y) \in \m_{X \times X}(\delta_X).
        \]
        In other words, $\sim_X$ is $\m_{X \times X}(\delta_X)$.
    
        \item (On morphisms:) for every morphism $f\colon X\to Y$ in $\C$, $M'(f)$ maps $[x]_X\in M'(X)$ to $[M(f)(x)]_Y\in M'(Y)$.

    \end{itemize}
    
    The fact that $\sim_X$ is an equivalence relation follows from the fact that, 
    applying $\m$ to the inequalities in \cref{def:equality}\eqref{i:def=1} (reflexivity),
    \cref{r:delta-symm} (symmetry), and
    \cref{r:trans} (transitivity), and using the naturality of $\m$, we obtain, respectively,
    \begin{enumerate}
        \item $\top_{\pws(M(X))}\subseteq \pws(\Delta_{M(X)})(\m_{X\times X}(\delta_X))$,
        \item $\m_{X\times X}(\delta_X)\subseteq \pws (\ple{\pr_2,\pr_1})(\m_{X\times X}(\delta_X))$,
        \item $\pws(\ple{\pr_1,\pr_2})(\m_{X\times X}(\delta_X))\cap\pws(\ple{\pr_2,\pr_3})(\m_{X\times X}(\delta_X))\subseteq \pws(\ple{\pr_1,\pr_3})(\m_{X\times X}(\delta_X))$.
    \end{enumerate}
    Unfolding the definitions, these state exactly that ${\sim_X}=\m_{X\times X}(\delta_X)$ is reflexive, symmetric and transitive.
    
    Let us prove that $M'$ is well defined on morphisms.
    Let $f \colon X \to Y$ be a morphism, and let $x,y \in M(X)$ with $x \sim_X y$ (i.e., $(x,y) \in \m_{X \times X}(\delta_X)$).
    By \cref{rmk:aeq} and the naturality of $\m$ at the morphism $f \times f \colon X \times X \to Y \times Y$, we have $(x,y) \in \m_{X \times X}(\P(f\times f)(\delta_Y))=M(f\times f)^{-1}[\m_{Y\times Y}(\delta_Y)]$,
    i.e.\
    \[
    \bigl(M(f)(x),M(f)(y)\bigr)\in\m_{Y\times Y}(\delta_Y),
    \]
    i.e.\ $M(f)(x)\sim_Y M(f)(y)$,
    as desired.
    
    Thus we obtain a well-defined functor
    \[
    M'\coloneqq M(-)_{/\sim_{-}}\colon\C\to\Set.
    \]
    We now show that $M'$ preserves finite products. For the binary case, we show that the bijection
    \[
    M(X\times Y)
    \cong M(X)\times M(Y)
    \]
    descends to a bijection on the quotient
    \[
    M(X\times Y)/{\sim_{X\times Y}} \cong  M(X)/{\sim_{X}}\times M(Y)/{\sim_{Y}}.
    \]
    For all $x,x'\in M(X)$ and $y,y'\in M(Y)$, we have
    \begin{align*}
        &(x,y)\sim_{X\times Y}(x',y') \\
        &\iff (x,y,x',y')\in \m_{X\times Y\times X\times Y}(\delta_{X\times Y})\\
        &\iff (x,y,x',y')\in \m_{X\times Y\times X\times Y}\bigl(\P(\ple{\pr_1,\pr_3})(\delta_X)\land \P(\ple{\pr_2,\pr_4})(\delta_Y)\bigr)&&\text{(by Rem.~\ref{rmk:aeq-3-conv})}\\
        &\iff (x,y,x',y')\in \m_{X\times Y\times X\times Y}\bigl(\P(\ple{\pr_1,\pr_3})(\delta_X)\bigr)\cap\m_{X\times Y\times X\times Y}\bigl( \P(\ple{\pr_2,\pr_4})(\delta_Y)\bigr)\\
        &\iff (x,y,x',y')\in \pws(\ple{\pr_1,\pr_3})(\m_{X\times X}(\delta_X))\cap \pws(\ple{\pr_2,\pr_4})(\m_{Y\times Y}(\delta_Y))&&\text{(by nat.\ of $\m$)}\\
        &\iff (x,x')\in \m_{X\times X}(\delta_X)\text{ and }(y,y')\in\m_{Y\times Y}(\delta_Y)\\
        &\iff x\sim_X x'\text{ and }y\sim_Y y',
    \end{align*}
    as desired. This settles the binary case.
    The nullary case is immediate.
    
    We define the component at $X\in\C$ of the natural transformation $\m'\colon \P\to\pws\circ {M'}\op$ as
        \begin{align*}
        \m'_X \colon \P(X) & \longrightarrow \pws(M'(X))\\
        \alpha & \longmapsto \bigl\{[x]_{X} \mid x\in \m_X(\alpha)\bigr\}.
    \end{align*}
    We prove that the condition $x \in \m_X(\alpha)$ appearing in the definition of $\m'_X(\alpha)$
    does not depend on the choice of $x$ as a representative of its $\sim_X$-equivalence class.
    For every $X\in \C$, $\alpha\in\P(X)$, $x,y\in M( X)$ such that $x \in \m_X(\alpha)$ and $x\sim_X y$ we have 
    \begin{align*}
        (x,y)\in \pws(\pr_1)(\m_X(\alpha))\cap\m_{X\times X}(\delta_X)\equalexpl{by nat.\ of $\m$ at\\  $\pr_1\colon X\times X\to X$}\m_{X\times X}(\P(\pr_1)(\alpha)\land \delta_X)\subseteqexpl{Def.~\ref{def:equality}\eqref{i:def=2}} \m_{X\times X}(\P(\pr_2)(\alpha))\equalexpl{by nat.\ of $\m$ at\\ $\pr_2\colon X\times X\to X$}\pws(\pr_2)(\m_X(\alpha)),
    \end{align*}
    i.e.\ $y\in \m_X(\alpha)$, as desired.

    For every $X \in \C$, the function $\m'_X$ preserves the negation because, for every $\alpha \in \P(X)$ and $x\in M(X)$,
    \[
        [x]_X \in \m'_X(\lnot\alpha) \iff x\in\m_X(\lnot\alpha)
        \iff x\notin\m_X(\alpha)
        \iff [x]_X \notin \m'_X(\alpha).
    \]
    Similarly, it is easily seen that $\m'_X$ preserves $\top$ and $\land$.
    
    The naturality of $\m'$ holds because, for all $X, X' \in \C$, every $\alpha\in \P(X)$, every morphism $f\colon X'\to X$ in $\C$ and every $x\in M( X')$,
    \begin{align*}
        [x]_{X'}\in \m'_{X'}(\P(f)(\alpha))
        & \iff x\in \m_{X'}(\P(f)(\alpha))\\
        & \iff x\in\pws(M(f))(\m_{X}(\alpha))&&\text{(by nat.\ of $\m$ at $f\colon X'\to X$)}\\
        & \iff M(f)(x) \in \m_{X}(\alpha)\\
        & \iff [M(f)(x)]_{X} \in \m'_{X}(\alpha)\\
        & \iff M'(f)([x]_{X'} )\in \m'_{X}(\alpha)\\
        & \iff [x]_{X'} \in \pws (M'(f)) (\m'_{X}(\alpha)).
    \end{align*}

    Let us prove that $\m'$ preserves the existential quantifier.
    For $\alpha\in\P(X\times Y)$,
    we shall prove
    \[
    \m'_X(\ex{Y}{X}\alpha)=\bigl\{[x]_X \in M'(X) \mid \text{there is }[y]_Y\in M'(Y)\text{ s.t.\ } ([x]_X,[y]_Y)\in \m'_{X\times Y}(\alpha)\bigr\},
    \]
    i.e.
    \begin{equation*}
        \bigl\{[x]_X \mid x\in \m_X(\ex{Y}{X}\alpha)\bigr\} = \bigl\{[x]_X \mid \text{there is }[y]_Y\text{ s.t.\ } (x,y)\in\m_{X\times Y}(\alpha)\bigr\}.
    \end{equation*}
    This holds because, for every $x \in M(X)$,
    \begin{align*}
        x\in \m_X(\ex{Y}{X}\alpha)
        &\iff x\in \ex{M(Y)}{M(X)}\m_{X\times Y}(\alpha) &&\text{(since $\m$ preserves $\exists$)}\\
        &\iff \text{there is $y\in M(Y)$ s.t.\ } (x,y)\in \m_{X\times Y}(\alpha)\\
        &\iff \text{there is $[y]_Y\in M'(Y)$ s.t.\ } ([x]_X,[y]_Y)\in \m'_{X\times Y}(\alpha).
    \end{align*}
    
    Next, we show that $\m'$ preserves equality. For every $X\in\C$ we have:
    \begin{align*}
    \m'_{X\times X}(\delta_X)&= \bigl\{[(x,y)]_{X\times X}\mid (x,y)\in\m_{X\times X}(\delta_{X})\bigr\}\\
    &=\bigl\{([x]_X,[y]_X)\mid (x,y)\in\m_{X\times X}(\delta_{X})\bigr\}\\
    &=\bigl\{([x]_X,[y]_X)\mid x \sim_X y\bigr\}\\
    &=\Delta_{M'(X)}.
    \end{align*}
    
    This concludes the proof that $(M',\m')$ is a model of $\P$ (as an elementary first-order Boolean doctrine).
    
    To conclude, we observe that, for every $\alpha\in\P(\tmn)$,
    \[
        (M, \m) \vDash \alpha \iff \m_\tmn(\alpha)\text{ is the singleton}\iff \m'_\tmn(\alpha)\text{ is the singleton}\iff (M', \m') \vDash \alpha. \qedhere
    \]
\end{proof}

\begin{theorem}[Order separation via models]
\label{t:Godel_elem}
    Let $\P$ be an elementary first-order Boolean doctrine over a small category.
    If $\alpha \nleq \beta$ in $\P(\tmn)$, then $\P$ has a model $(M, \m)$ such that $(M, \m) \vDash \alpha$ and $(M, \m) \nvDash \beta$.
\end{theorem}

\begin{proof}
    The statement follows from its non-elementary version (\cref{t:Godel}), in virtue of \cref{t:quotient}.
\end{proof}

We say that an elementary first-order Boolean doctrine $\P$ is \emph{consistent} if it is so as a first-order Boolean doctrine, i.e.\ if $\top_{\P(\tmn)} \neq \bot_{\P(\tmn)}$.

\begin{corollary}[G\"odel's completeness for elementary first-order Boolean doctrines]\label{c:Godel_elem}
	Every consistent elementary first-order Boolean doctrine over a small category has a model.
\end{corollary}

\begin{proof}
    The statement follows from its non-elementary version (\cref{c:Godel}), in virtue of \cref{t:quotient}.
\end{proof}

\subsection{Generalization to open formulas and pointed models}
\begin{definition}[Model at an object of an elementary first-order Boolean doctrine] \label{d:bool-mod-at-S-elementary}
    Let $\P \colon \C\op \to  \BA$ be an elementary first-order Boolean doctrine, and let $S \in \C$.
    An \emph{$S$-pointed model of $\P$} is a triple $(M,\m,s)$ where $(M,\m)\colon \P \to \pws$ is a model of $\P$ (i.e.\ an elementary first-order Boolean doctrine morphism from $\P$ to the subset doctrine $\pws$) and $s \in M(S)$.
\end{definition}

\begin{remark}[{\cite[Prop.~5.4 and Thm.~6.3(i)]{GuffConst}}]\label{r:const_aeq}
    Let $\P$ be an elementary first-order Boolean doctrine over $\C$, let $S\in\C$, and let $(L_S,\mathfrak{l}_S)\colon \P\to\P_S$ be the construction that freely adds a constant of type $S$ to $\P$ as in \cref{ss:const}. Then $\P_S$ is an elementary first-order Boolean doctrine, and $(L_S,\mathfrak{l}_S)$ is an elementary first-order Boolean doctrine morphism. Moreover, the universal property of $(L_S,\mathfrak{l}_S)$ and $\id_S\colon\tmn\rightsquigarrow S$ (\cref{t:univ-prop}) respects the elementary structure in the following sense:
     \begin{quote}
        For every elementary first-order Boolean doctrine $\mathbf{R}\colon\cat{D}\op\to\BA$, elementary first-order Boolean doctrine morphism $(M,\m)\colon \P\to \mathbf{R}$ and morphism $c \colon \tmn_\cat{D} \to M(S)$ in $\cat{D}$, there is a unique elementary first-order Boolean doctrine morphism $(N,\n) \colon \P_S \to \mathbf{R}$ with $(M,\m)=(N,\n)\circ(L_S,\mathfrak{l}_S)$ and $N(\id_S\colon \tmn \rightsquigarrow S)=c$.
    \end{quote}
\end{remark}

As in \cref{n:semantic-conseq-at-S}, for an $S$-pointed model $(M, \m, s)$ of an elementary first-order Boolean doctrine $\P$ and $\alpha\in \P(S)$, we write $(M, \m, s) \vDash \alpha$ if $s \in \m_S(\alpha)$.

\begin{theorem} \label{t:quotient-elementary}
    Let $\P$ be an elementary first-order Boolean doctrine over $\C$, and let $S \in \C$.
    Let $(M, \m, s)$ be an $S$-pointed model of $\P$ as a (not necessarily elementary) first-order Boolean doctrine.
    There is an $S$-pointed model $(M', \m', s')$ of $\P$ as an \emph{elementary} first-order Boolean doctrine such that, for every $\alpha \in \P(S)$, $(M, \m,s) \vDash \alpha$ if and only if $(M', \m',s') \vDash \alpha$.
\end{theorem}

\begin{proof}
    By \cref{r:const_aeq}, the first-order Boolean doctrine $\P_S$ is elementary.
    
    The (not necessarily elementary) $S$-pointed model $(M, \m, s)$ of $\P$ corresponds to a (not necessarily elementary) model $(N,\n)$ of $\P_S$ in light of \cref{t:univ-prop}. By \cref{t:quotient}, we define a model $(N',\n')$ of the elementary first-order Boolean doctrine $\P_S$, which in turn corresponds via \cref{r:const_aeq} to an $S$-pointed model $(M',\m',s')$ of the elementary first-order Boolean doctrine $\P$.
    
    To conclude, we observe that, for every $\alpha\in\P(S)$, 
    \[
    (M, \m,s) \vDash \alpha \iffexpl{by \cref{l:model-model-at-S}}  (N, \n) \vDash \alpha \iffexpl{by \cref{t:quotient}} (N', \n') \vDash \alpha \iffexpl{by \cref{l:model-model-at-S}}  (M', \m',s') \vDash \alpha.\qedhere
    \]
\end{proof}

\begin{corollary}[of \cref{t:Godel_elem}; order-separation via models, at any fiber]
    \label{c:godel-at-s_aeq}
    Let $\P$ be an elementary first-order Boolean doctrine over a small category $\C$.
    For all $S \in \C$ and $\alpha, \beta \in \P(S)$ such that $\alpha \nleq \beta$, there is an $S$-pointed model $(M, \m, s)$ of $\P$ such that $(M, \m, s) \vDash \alpha$ and $(M, \m, s) \nvDash \beta$.
\end{corollary}

\begin{proof}
    The statement follows from its non-elementary version (\cref{c:godel-at-s}), in virtue of \cref{t:quotient-elementary}.
\end{proof}

\section{Consequence: the type space functor is the Stone dual of the doctrine}
\label{s:type-space-functor}
\subsection{The space of types is the Stone dual of the algebra of sentences}

Stone duality says that the category $\BA$ of Boolean algebras and Boolean homomorphisms is dually equivalent to the category $\Stone$ of Stone spaces and continuous maps. (We recall that a Stone space is a compact space that is \emph{totally separated}, i.e., in which clopen subsets separate distinct points.)

Recall that a \emph{model} of a first-order Boolean doctrine $\P$ is a first-order Boolean doctrine morphism $(M,\m)$ from $\P$ to the subset doctrine  $\pws$.
\[
    \begin{tikzcd}
        \C\op && \Set\op \\
        \\
        & \BA
        \arrow["M\op", from=1-1, to=1-3]
        \arrow[""{name=0, anchor=center, inner sep=0}, "\P"', from=1-1, to=3-2]
        \arrow[""{name=1, anchor=center, inner sep=0}, "{\pws}", from=1-3, to=3-2]
        \arrow["\m", curve={height=-6pt}, shorten <=7pt, shorten >=7pt, from=0, to=1]
    \end{tikzcd}
\]
Moreover, recall from \cref{n:semantic-conseq} that, for $\alpha \in \P(\tmn)$, we write $(M, \m) \vDash \alpha$ if $\m_\tmn(\alpha)$ is the whole singleton $M(\tmn)$.

\begin{definition}\label{d:mod-p}
    Let $\P$ be a first-order Boolean doctrine. We define the class
    \[
    \Mod(\P)\coloneqq\{(M,\m)\mid (M,\m)\text{ is a model of $\P$}\}.
    \]
\end{definition}

We define the relation $\equiv$  on $\Mod(\P)$ of \emph{elementary equivalence}\footnote{We use the expression ``elementary equivalence'' in analogy with classical model theory, where two structures are elementarily equivalent if they satisfy the same first-order sentences.
Our terminology does not assume that $\P$ is elementary (i.e.\ equipped with equality, see \cref{def:equality}).} as follows:
\[
(M,\m) \equiv (M',\m') \iff \Big(\text{for all }\alpha\in\P(\tmn),\, (M,\m) \vDash \alpha\Leftrightarrow (M',\m') \vDash \alpha\Big).
\]

\begin{definition}[Space of types]\label{d:types}
    Let $\P$ be a first-order Boolean doctrine. 
    We define the set\footnote{While $\Mod(\P)$ may not be small, note that the quotient $\Mod(\P)/{\equiv}$ is a set: from the definition, we can see that it has at most $\pws(\P(\tmn))$ elements. Additionally, if the base category is small, \cref{t:types-are-ultrafilters} shows that $\Mod(\P)/{\equiv}$ is in bijection with the set of ultrafilters of $\P(\tmn)$.} of \emph{types of $\P$} as the quotient
    \[
    \Typ(\P)\coloneqq  {\Mod(\P)}/{\equiv}.
    \]
    We equip $\Typ(\P)$ with the topology generated by the sets of the form
    \begin{equation}\label{eq:base-topology}
        \llbracket \alpha \rrbracket \coloneqq \Bigl\{\bigl[(M, \m)\bigr] \mid  (M, \m) \vDash \alpha \Bigr\},    \qquad \qquad \alpha \in \P(\tmn).
    \end{equation}
    (The condition $(M, \m) \vDash \alpha$ does not depend on the choice of $(M, \m)$ as a representative of its equivalence class.)
\end{definition}

For a Boolean algebra $B$, let us denote by $\Ult(B)$ the set of ultrafilters of $B$. Moreover, recall that $\Ult(B)$ is a Stone space whose clopens are precisely the sets of the form
\[
\widehat{a} \coloneqq \{U \in \Ult(B) \mid a \in U\},   \qquad \qquad a \in B.
\]

\begin{theorem}[The space of types is the Stone dual of the algebra of sentences]\label{t:types-are-ultrafilters}
    Let $\P$ be a first-order Boolean doctrine over a small category $\C$.
    The space $\Typ(\P)$ is the Stone dual of the Boolean algebra $\P(\tmn)$, as witnessed by the homeomorphism:
    \begin{align*}
        \Typ(\P) & \longrightarrow \Ult(\P(\tmn))\\
        \bigl[(M, \m)\bigr] & \longmapsto \{\alpha \in \P(\tmn) \mid (M, \m) \vDash \alpha\}.
    \end{align*}
\end{theorem}

\begin{proof}
    By definition of the equivalence relation on $\Mod(\P)$, the set $\{\alpha \in \P(\tmn) \mid (M, \m) \vDash \alpha\}$ does not depend on the choice of $(M, \m)$ as a representative of its equivalence class. 
    It is an ultrafilter because, by definition of $\vDash$, it is the preimage of $\{\top\}$ under the Boolean homomorphism
    \[
    \P(\tmn)\xlongrightarrow{\m_\tmn}\pws(M(\tmn)) \cong 2.
    \]

    The function is injective by definition of the equivalence relation on $\Mod(\P)$.
    
    To prove surjectivity, let
    \[
    U \subseteq \P(\tmn)
    \]
    be an ultrafilter.
    Then $\P_{\!/U}\colon\C\op\to\BA$ is a first-order Boolean doctrine and $\P_{\!/U}(\tmn)=2$ (see \cref{r:quotient-doctrine-filter}).
    By G\"{o}del's completeness theorem (\cref{c:Godel}), $\P_{\!/U}$ has a model $(N,\n) \colon \P_{\!/U} \to \pws$.
    We define a model $(M,\m)$ of $\P$ as the composite
    \[
    \P 
    \xrightarrow{(\id_\C,\pi)}
    \P_{\!/U}
    \xrightarrow{(N,\n)}
    \pws.
    \]
    
    It remains to show that $U = \{\alpha \in \P(\tmn) \mid (M, \m) \vDash \alpha\}$.
    This holds because, for every $\alpha\in \P(\tmn)$, setting $\{*\} = M(\tmn) = N(\tmn)$
    \begin{align*}
        (M,\m)\vDash\alpha &\iff (N,\n)\circ(\id_\C,\pi) \vDash\alpha\\
        &\iff \n_\tmn(\pi_\tmn(\alpha))=\top_{\pws(\{*\})}\\
        &\iff  \pi_\tmn(\alpha)=\top_{{\P}_{\!/U}(\tmn)}&&\text{(since $\n_\tmn \colon \P_{\!/U}(\tmn) \to \pws(\{*\})$ is the isomorphism $2 \to 2$)}\\
        &\iff \alpha\in U.
    \end{align*}
    This proves that the function in the statement is surjective.
    
    Finally, we prove that this bijection is a homeomorphism. 
    For $\alpha \in \P(\tmn)$, the preimage of the basic clopen $ \widehat{\alpha} = \{U \in \Ult(\P(\tmn)) \mid \alpha \in U\}$ of $\Ult(\P(\tmn))$
    under the bijection is the set
    \begin{align*}
        \Bigl\{\bigl[(M, \m)\bigr] \mathrel{\big|} \bigl\{\beta \in \P(\tmn) \mid (M, \m) \vDash \beta\bigr\} \in \widehat{\alpha}\Bigr\} & = \Bigl\{\bigl[(M, \m)\bigr] \mathrel{\big|} \alpha \in \bigl\{\beta \in \P(\tmn) \mid (M, \m) \vDash \beta\bigr\}\Bigr\} \\
        & = \Bigl\{\bigl[(M, \m)\bigr] \mathrel{\big|} (M, \m) \vDash \alpha \Bigr\}\\
        &=\llbracket \alpha \rrbracket &&\text{(by \eqref{eq:base-topology})}.
    \end{align*}
    Thus, the preimage of the basis of clopens $\{ \widehat{\alpha} \mid \alpha\in\P(\tmn)\}$ of $\Ult(\P(\tmn))$ under the bijection in the statement is precisely the set of generators $\{\llbracket \alpha \rrbracket\mid \alpha\in\P(\tmn)\}$ of the topology of $\Typ(\P)$. Therefore, the bijection is a homeomorphism.
\end{proof}

In conclusion, G\"{o}del's completeness theorem has as a consequence that the space of types is the Stone dual of the Boolean algebra of closed formulas.

\begin{remark}[Clopens of the space of types] \label{r:clopens-of-space-of-types}
    It follows from the proof of \cref{t:types-are-ultrafilters} that, for $\P$ a first-order Boolean doctrine over a small category, the clopens of $\Typ(\P)$ are precisely the sets of the form
    \begin{equation*}
        \llbracket \alpha \rrbracket \coloneqq \Bigl\{\bigl[(M, \m)\bigr] \mid  (M, \m) \vDash \alpha \Bigr\},    \qquad \qquad \alpha \in \P(\tmn).
    \end{equation*}
\end{remark}

\subsection{The space of \texorpdfstring{$S$}{S}-types is the Stone dual of the algebra of \texorpdfstring{$S$}{S}-formulas}\label{ss:s-typ}

Recall that an $S$-pointed model of a first-order Boolean doctrine $\P \colon \C\op \to  \BA$ is a triple $(M,\m,s)$ with $(M,\m)\colon \P \to \pws$ a model of $\P$ and $s \in M(S)$.
Moreover, recall that, for $\alpha \in \P(S)$, we write $(M, \m, s) \vDash \alpha$ if $s \in \m_S(\alpha)$.

\begin{definition}
    Let $\P$ be a first-order Boolean doctrine over $\C$ and let $S\in\C$. We define the class
    \[
    \Mod_S(\P)\coloneqq\bigl\{(M,\m,s)\mid (M,\m,s)\text{ is an $S$-pointed model of $\P$}\bigr\}.
    \]
\end{definition}

We define the relation $\equiv$ on $\Mod_S(\P)$ of \emph{elementary equivalence} as follows:
\[
(M,\m,s)\equiv (M',\m',s') \iff \Big(\text{for all }\alpha\in\P(S),\, (M,\m,s) \vDash \alpha\Leftrightarrow (M',\m',s') \vDash \alpha\Big).
\]

\begin{definition}[Space of types at an object]\label{d:types-at-S}
    Let $\P$ be a first-order Boolean doctrine over $\C$, and $S\in\C$. 
    We define the set of \emph{$S$-types of $\P$} as the quotient
    \[
    \Typ_S(\P)\coloneqq  {\Mod_S(\P)}/{\equiv}.
    \]
    We equip $\Typ_S(\P)$ with the topology generated by the sets of the form
    \begin{equation*}
        \llbracket \alpha \rrbracket \coloneqq \Bigl\{\bigl[(M, \m,s)\bigr] \mid  (M, \m,s) \vDash \alpha \Bigr\},    \qquad \qquad \alpha \in \P(S).
    \end{equation*}
    (The condition $(M, \m, s) \vDash \alpha$ does not depend on the choice of $(M, \m,s)$ as a representative of its equivalence class.)
\end{definition}

\begin{remark} \label{r:bijection}
    Let $\P$ be a first-order Boolean doctrine over $\C$, let $S\in\C$, and let $(L_S,\mathfrak{l}_S)\colon \P\to\P_S$ be the construction that freely adds a constant of type $S$ to $\P$. By \cref{t:univ-prop}, the following is a bijection
    \begin{align*}
        \Mod(\P_S) &\longrightarrow \Mod_S(\P) \\
        (N,\n)&\longmapsto \Big((N,\n)\circ (L_S,\mathfrak{l}_S), N(\id_S\colon \tmn \rightsquigarrow S)\Big).
    \end{align*}
    By \cref{l:model-model-at-S}, this assignment descends to a bijection on the quotients
    \[
    \Typ(\P_S)\longrightarrow\Typ_S(\P),
    \]
    and moreover, the basis $\{\llbracket\alpha\rrbracket\mid \alpha\in\P_S(\tmn)\}$ of the topology on $\Typ(\P_S)$ defined in \cref{d:types} is mapped precisely to the basis $\{\llbracket\alpha\rrbracket\mid \alpha\in\P(S)\}$ defined in \cref{d:types-at-S}, and thus the bijection $\Typ(\P_S)\to\Typ_S(\P)$ is a homeomorphism.
\end{remark}

\begin{theorem}[The space of $S$-types is the Stone dual of the algebra of $S$-formulas]\label{t:S-types-are-ultrafilters}
    Let $\P$ be a first-order Boolean doctrine over a small category $\C$ and let $S\in\C$.
    The space $\Typ_S(\P)$ is the Stone dual of the Boolean algebra $\P(S)$, as witnessed by the homeomorphism
    \begin{align*}
        \Typ_S(\P) & \longrightarrow \Ult(\P(S))\\
        \bigl[(M, \m,s)\bigr] & \longmapsto \bigl\{\alpha \in \P(S) \mid (M, \m,s) \vDash \alpha\bigr\}.
    \end{align*}
\end{theorem}

\begin{proof}
    Modulo the homeomorphism in \cref{r:bijection}, this is the instantiation of \cref{t:types-are-ultrafilters} to $\P_S$.
\end{proof}

\begin{remark}[Clopens of the space of types]\label{r:clopens-of-space-of-types-at-object}
    By \cref{r:clopens-of-space-of-types}, for $\P$ a first-order Boolean doctrine over a small category, the clopens of $\Typ_S(\P)$ are precisely the sets of the form
    \begin{equation*}
        \llbracket \alpha \rrbracket \coloneqq \Bigl\{\bigl[(M, \m,s)\bigr] \mid  (M, \m,s) \vDash \alpha \Bigr\},    \qquad \qquad \alpha \in \P(S).
    \end{equation*}
\end{remark}

\subsection{The type space functor is the Stone dual of the doctrine}\label{ss:types-dual-doctrine}
Let $\P \colon \C\op \to \BA$ be a first-order Boolean doctrine.
For every morphism $f \colon S \to S'$ in $\C$, we have a continuous function
\begin{align*}
    \Typ_f(\P) \colon \Typ_S(\P) & \longrightarrow \Typ_{S'}(\P)\\
    \bigl[(M, \m, s)\bigr] & \longmapsto \Bigl[\bigl(M, \m, M(f)(s)\bigr)\Bigr].
\end{align*}
To prove that this function is indeed well-defined and continuous, we use the following.
\begin{remark}\label{r:typ-f}
    Let $\P$ be a first-order Boolean doctrine over $\C$ and $f\colon S\to S'$ a morphism in $\C$. For every $(M,\m,s)\in\Mod_S(\P)$ and $\beta\in\P(S')$,
    \begin{align*}
        (M,\m,M(f)(s))\vDash \beta &\iff M(f)(s)\in\m_{S'}(\beta)\\
        &\iff s\in M(f)^{-1}[\m_{S'}(\beta)]\\
        &\iff s\in \m_S\bigl(\P(f)(\beta)\bigr) &&\text{(by naturality of $\m$ at $f \colon S \to S'$)}\\
        &\iff (M,\m,s)\vDash \P(f)(\beta).
    \end{align*}
\end{remark}

We show that the function $\Typ_f(\P)$ is well defined.
Let $(M,\m,s), (M',\m',s') \in \Mod_S(\P)$ be such that $(M,\m,s)\equiv(M',\m',s')$; we shall prove $(M,\m,M(f)(s))\equiv(M',\m',M'(f)(s'))$. For every $\beta\in \P(S')$,
\begin{align*}
    \bigl(M,\m,M(f)(s)\bigr)\vDash \beta 
    &\iff (M,\m,s)\vDash \P(f)(\beta) &&\text{(by \cref{r:typ-f})}\\
    &\iff (M',\m',s')\vDash \P(f)(\beta) &&\text{(since $(M,\m,s)\equiv(M',\m',s')$)}\\
    &\iff \bigl(M',\m',M'(f)(s')\bigr)\vDash \beta &&\text{(by \cref{r:typ-f})},
\end{align*}
as desired.

To prove that it is continuous, we observe that the preimage of a basic open is a basic open: indeed, for every $\beta \in \P(S')$, we have
\begin{align*}
    \Typ_f(\P)^{-1}\bigl[\llbracket \beta \rrbracket\bigr]
    & = \Bigl\{\bigl[(M, \m, s)\bigr] \in \Typ_S(\P) \mid \bigl(M, \m, M(f)(s)\bigr) \vDash \beta \Bigr\}\\
    & = \Bigl\{\bigl[(M, \m, s)\bigr] \in \Typ_S(\P) \mid (M, \m, s) \vDash \P(f)(\beta)\Bigr\} && \text{(by \cref{r:typ-f})}\\
    & =\llbracket \P(f)(\beta)\rrbracket.
\end{align*}

\begin{definition}[The type space functor]
    Let $\P \colon \C\op \to \BA$ be a first-order Boolean doctrine, with $\C$ small.
    We let $\Typ_{-}(\P) \colon \C \to \Stone$ denote the functor that:
    \begin{itemize}
        \item (Objects) to each object $S$ assigns the Stone space 
        \[
        \Typ_S(\P) \coloneqq \Mod_S(\P)/{\equiv},
        \]
        defined in \cref{d:types-at-S},
        
        \item (Morphisms) to each morphism $f \colon S \to S'$ assigns the continuous function
        \begin{align*}
            \Typ_f(\P) \colon \Typ_S(\P) & \longrightarrow \Typ_{S'}(\P)\\
        [(M, \m, s)] & \longmapsto [(M, \m, M(f)(s))].
        \end{align*}
    \end{itemize}
\end{definition}

In the following, we observe that ${\Typ_f(\P)}$ is the Stone dual of the Boolean homomorphism $\P(f) \colon \P(S') \to \P(S)$, modulo the homeomorphism in \cref{t:S-types-are-ultrafilters}.

\begin{theorem}[The type space functor is the Stone dual of the doctrine]
    Let $\P \colon \C\op \to \BA$ be a first-order Boolean doctrine, with $\C$ small.
    The type space functor 
    \[
    \Typ_{-}(\P) \colon \C \longrightarrow \Stone
    \]
    is the Stone dual of $\P \colon \C\op \to \BA$, i.e., it is naturally isomorphic to the (opposite functor of the) composite of $\P$ and Stone duality $\BA \cong \Stone\op$.
\end{theorem}

\begin{proof}
    The natural isomorphism is componentwise defined as the homeomorphism in \cref{t:S-types-are-ultrafilters}:
    \begin{align*}
        \Typ_S(\P) & \longrightarrow \Ult(\P(S))\\
        \bigl[(M, \m,s)\bigr] & \longmapsto \bigl\{\alpha \in \P(S) \mid (M, \m,s) \vDash \alpha\bigr\}.
    \end{align*}
    We are left to check that it is indeed a natural transformation, i.e., that, for every morphism $f \colon S \to S'$ in $\C$, the following diagram commutes.
    \[
    \begin{tikzcd}
    	 {\Typ_S(\P)}&  {\Ult(\P(S))}\\
    	 {\Typ_{S'}(\P)}& {\Ult(\P(S'))}
    	\arrow["\cong", from=1-1, to=1-2]
    	\arrow["{\Typ_f(\P)}"', from=1-1, to=2-1]
    	\arrow["{\Ult(\P(f)) = \P(f)^{-1}[-]}",  from=1-2, to=2-2]
    	\arrow["\cong", from=2-1, to=2-2]
    \end{tikzcd}
    \]
    Let $[(M,\m,s)]\in\Typ_S(\P)$; we shall prove
    \[
    \P(f)^{-1}[\{\alpha\in\P(S)\mid (M,\m,s)\vDash\alpha\}]=\{\beta\in\P(S')\mid (M,\m,M(f)(s))\vDash\beta\},
    \]
    but this trivially follows from \cref{r:typ-f}.
\end{proof}

\subsection{The case with equality}
In this subsection, we show that the results of the previous part of this section remain valid in the setting with equality:
\begin{quote}
    The (appropriate) type space functor is (still) the Stone dual of the doctrine.
\end{quote}
To this end, we adapt the definition of a type of a first-order Boolean doctrine $\P$ to the case in which $\P$ is elementary, and we only consider models of $\P$ as an elementary first-order Boolean doctrine.
The key observation is that this restriction on models does not affect the resulting space of types. Indeed, by \cref{t:quotient}, every model of $\P$ as a (not necessarily elementary) first-order Boolean doctrine is elementarily equivalent to a model of $\P$ as an \emph{elementary} first-order Boolean doctrine.

Recall that a \emph{model} of an elementary first-order Boolean doctrine $\P$ is an elementary first-order Boolean doctrine morphism $(M,\m)$ from $\P$ to the subset doctrine $\pws$.
Moreover, recall that, given a model $(M, \m)$ of an elementary first-order Boolean doctrine $\P$ and $\alpha \in \P(\tmn)$, we write $(M, \m) \vDash \alpha$ if $\m_\tmn(\alpha)$ is the whole singleton $M(\tmn)$.

\begin{definition}
    Let $\P$ be an elementary first-order Boolean doctrine. We define the class
    \[
    \Mod_=(\P)\coloneqq\{(M,\m)\mid (M,\m)\text{ is a model of $\P$}\}.
    \]
\end{definition}

We use the symbol $=$ in $\Mod_=(\P)$ to distinguish it from $\Mod(\P)$, by which we still mean the class of (not necessarily elementary) models of $\P$ as in \cref{d:mod-p}.

The relation of elementary equivalence on $\Mod(\P)$
\[
(M,\m) \equiv (M',\m') \iff \Big(\text{for all }\alpha\in\P(\tmn),\, (M,\m) \vDash \alpha\Leftrightarrow (M',\m') \vDash \alpha\Big)
\]
restricts to an equivalence relation on $\Mod_=(\P)$.

\begin{definition}[Space of types]\label{d:types-aeq}
    Let $\P$ be an elementary first-order Boolean doctrine. 
    We define the set of \emph{types of $\P$} as the quotient
    \[
    \Typ_=(\P)\coloneqq  {\Mod_=(\P)}/{\equiv}.
    \]
    We equip $\Typ_=(\P)$ with the topology generated by the sets of the form
    \[
        \llbracket \alpha \rrbracket \coloneqq \Bigl\{\bigl[(M, \m)\bigr] \mid  (M, \m) \vDash \alpha \Bigr\},    \qquad \qquad \alpha \in \P(\tmn).
    \]
\end{definition}

\begin{remark} \label{r:descends-to-homeo}
    Let $\P$ be an elementary first-order Boolean doctrine. By \cref{t:quotient}, every model of $\P$ as a (not necessarily elementary) first-order Boolean doctrine is elementarily equivalent to a model of $\P$ as an \emph{elementary} first-order Boolean doctrine.
    Therefore, the inclusion 
    \[
    \Mod_=(\P) \longhookrightarrow \Mod(\P)
    \]
    descends to a \emph{homeomorphism}(!) on the quotient
    \[
    \Typ_=(\P) \longrightarrow \Typ(\P).
    \]
\end{remark}

An analogous fact is true for pointed models, by \cref{t:quotient-elementary}.

Therefore, all results in \cref{ss:s-typ,ss:types-dual-doctrine} remain valid if $\P$ is an \emph{elementary} first-order Boolean doctrine and one restricts to models of $\P$ as an \emph{elementary} first-order Boolean doctrine.

\makeatletter
\begingroup
\let\addcontentsline\@gobblethree
\section*{Acknowledgments}

Marco Abbadini thanks Sean Moss for a discussion about G\"odel's completeness theorem for first-order Boolean doctrines, which gave some ideas to get a simpler proof.

\subsection*{Funding}
Marco Abbadini was funded by UK Research and Innovation (UKRI) under the UK government’s Horizon Europe funding guarantee (grant number EP/Y015029/1, Project ``DCPOS'') during his affiliation at the University of Birmingham, and by an FSR Incoming Postdoctoral Fellowship during his affiliation at the Université catholique de Louvain.

Francesca Guffanti was funded by the SHINE program of the French National Research Agency (ANR) under the project ``GULI'' (Grandeurs et Unités pour les Langages Informatiques), grant number ANR-22-EXES-0017.

\endgroup
\makeatother

\bibliographystyle{apalike}
\bibliography{Biblio}

\end{document}